\documentclass{article}

\usepackage{comment}
\usepackage[utf8]{inputenc}
\usepackage{amsmath,amsfonts,amsthm, amssymb}
\usepackage[dvipsnames]{xcolor}

\usepackage{fancyvrb} 

\usepackage{biblatex}
\usepackage{hyperref}

\usepackage{xurl}
\usepackage{todonotes} 

\theoremstyle{plain}
\newtheorem{lem}{Lemma}[section]

\newtheorem{thm}{Theorem}[section]
\newtheorem{cor}[lem]{Corollary}
\newtheorem{definition}{Definition}

\usepackage{multirow}

\newcounter{InlineToDo}

\title{From discrete to continuous: Monochromatic 3-term arithmetic progressions}

\author{Torin Greenwood\thanks{Department of Mathematics, North Dakota State University, Fargo, ND USA, \href{mailto:torin.greenwood@ndsu.edu}{torin.greenwood@ndsu.edu}}, Jonathan Kariv\thanks{Isazi Consulting, Johannesburg, South Africa, \href{mailto:jkariv@isaziconsulting.co.za}{jkariv@isaziconsulting.co.za}}, Noah Williams\thanks{Department of Mathematical Sciences, Appalachian State University, Boone, NC USA, \href{mailto:williamsnn@appstate.edu}{williamsnn@appstate.edu}}}

\date{December 31, 2022}

\DeclareMathOperator{\AP}{AP}

\newcommand{\intLength}{N}
\newcommand{\progLength}{\kappa}
\newcommand{\numBlocks}{n}
\newcommand{\blockApproxVar}{\ell}

\begin{document}

\maketitle

\begin{abstract}
    We prove a known 2-coloring of the integers $[\intLength] := \{1,2,3,...,\intLength\}$ minimizes the number of monochromatic arithmetic 3-progressions under certain restrictions. A monochromatic arithmetic progression is a set of equally-spaced integers that are all the same color. Previous work by Parrilo, Robertson and Saracino conjectured an optimal coloring for large $\intLength$ that involves $12$ colored blocks. Here, we prove that the conjecture is optimal among anti-symmetric colorings with $12$ or fewer colored blocks. We leverage a connection to the coloring of the continuous interval $[0,1]$ used by Parrilo, Robertson, and Saracino as well as by Butler, Costello and Graham. Our proof identifies classes of colorings with permutations, then counts the permutations using mixed integer linear programming.
\end{abstract}

\section{Introduction}
Consider coloring each of the integers in $[\intLength]$ with one of $r$ colors.  A $\progLength$-term arithmetic progression is any subset of $\progLength$ equally-spaced integers, denoted a $\progLength$-AP.  An arithmetic progression is monochromatic if every term is colored the same color.  Can we color $[\intLength]$ in a way that avoids all monochromatic $\progLength$-APs?  A classic result is van der Waerden's Theorem:

\begin{thm}[van der Waerden, \cite{vdW:1927}]\label{VDW}
For any integers $r, \progLength \geq 1$, there exists a number $\intLength$ such that every $r$-coloring of $[\intLength]$ has a monochromatic $\progLength$-AP.
\end{thm}

Given that monochromatic $\progLength$-APs are guaranteed to exist when enough numbers are colored, we ask a refined question: what is the minimum number of monochromatic $\progLength$-APs that could exist?  To be more precise, define $\mathcal{C}_r(\intLength)$ to be the set of $r$-colorings of $[\intLength]$.  For any $c \in \mathcal{C}_r(\intLength)$, let $m_\progLength(c)$ be the number of monochromatic $\progLength$-APs induced by $c$. Finally, let $\AP_\progLength(\intLength)$ be the total number of $\progLength$-APs in $[\intLength]$, regardless of whether they are monochromatic or not.  Then, we look at
\[
P_{r, \progLength}(\intLength) := \min_{c \in \mathcal{C}_r(\intLength)} \frac{m_\progLength(c)}{\AP_\progLength(\intLength)}.
\]
The focus of this paper is to examine the minimum for monochromatic $3$-APs within $2$-colorings, $P(\intLength) := P_{2, 3}(\intLength)$.  In 1999, Ron Graham proposed that $\lim_{n \to \infty} P(n) = \beta$ for some constant, $\beta$, and offered a \$100 prize for finding $\beta$.  Originally, it was not clear whether colorings could perform better than random in the long run: for large values of $\intLength$, is it possible to color $[\intLength]$ so the probability that a randomly selected $3$-AP is monochromatic is less than $(1/2)^3 + (1/2)^3 = 1/4$? It is notable that the analogous question for $2$-colorings of $\mathbb{Z}_p$ is answered negatively for $p$ prime. Indeed, Lu and Peng \cite{Lu:2012} show that for a given $2$-coloring of $\mathbb{Z}_p$, the fraction of $3$-APs that are monochromatic  depends only on the fraction of each color present in the coloring.

For our question concerning $2$-colorings of $[\intLength]$,  Parrilo et al. \cite{Parrilo:2008} and Butler et al. \cite{Butler:2010} verified independently but nearly simultaneously that it is possible to do better than random, and they found upper and lower bounds for the minimum monochromatic APs.
The upper bound was attained through simulating good colorings and finding one that performed well.  They landed on the following $12$-block coloring:
\begin{center}
\includegraphics[width=0.9\textwidth]{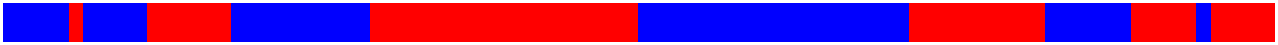}
\end{center}
Explicitly, when coloring $[\intLength]$, the blocks would be approximately of the following sizes:
\begin{equation}
\left(\frac{28\intLength}{548}, \frac{6\intLength}{548}, \frac{28\intLength}{548}, \frac{37\intLength}{548}, \frac{59\intLength}{548}, \frac{116\intLength}{548}, \frac{116\intLength}{548}, \frac{59\intLength}{548}, \frac{37\intLength}{548}, \frac{28\intLength}{548}, \frac{6\intLength}{548}, \frac{28\intLength}{548} \right) \label{eq:ParriloColoring}
\end{equation}
Due to this coloring, $P(\intLength) \leq \frac{117}{548} + o(1)$.  Note that this coloring is \emph{anti-symmetric}: the left half of the coloring is a mirror image of the right half but uses opposite colors.  In \cite{Butler:2010}, Butler et al. performed many computer simulations using genetic algorithms to find the optimal coloring, and noted that this same $12$-block coloring consistently appeared regardless of the seed coloring with which they started.  They noted that a remaining challenge would be to analyze the case of rapidly alternating colorings.

The goal of this paper is to show that as $\intLength \to \infty$, the $2$-coloring of $[\intLength]$ that has alternating color blocks with sizes given in Equation \eqref{eq:ParriloColoring} is globally optimal among anti-symmetric colorings with at most $12$ blocks.  As far as the authors are aware, this is the first result of optimality under any restrictions.  Here, we let $\tilde{\mathcal{C}}_2(\intLength)$ be the $2$-colorings of $[\intLength]$ that are anti-symmetric and have at most $12$ contiguous segments of red or blue.  Then, define
\[
\tilde{P}(\intLength) = \min_{c \in \tilde{\mathcal{C}}_2(\intLength)} \frac{m_3(c)}{\AP_3(\intLength)}.
\]
Our main result is as follows:
\begin{thm} \label{thm:MainResult}
Consider coloring each integer in $[\intLength]$ with either red or blue such that the coloring is anti-symmetric and has at most $12$ contiguous blocks.  Then, as $\intLength$ increases the minimum possible fraction of arithmetic progressions approaches $\frac{117}{548}$.  That is, $\lim_{\intLength \to \infty} \tilde{P}(\intLength) = \frac{117}{548}$.
\end{thm}

Below, we provide a proof sketch that outlines the sections in the paper.

\begin{proof}[Sketch of proof]
First, we will convert from discrete colorings of $[\intLength]$ to continuous colorings of $[0, 1]$ with at most $12$ contiguous segments, referred to as \emph{block colorings.}  After restricting the number of color changes that can occur within a coloring, it turns out that optimizing the discrete colorings is the same as optimizing the continuous colorings, as described rigorously in Lemma \ref{cor:DiscToContinuousFinal}.

When switching to the continuous realm, we let a continuous coloring be a function $c: [0, 1] \to \{0, 1\}$, where $0$ and $1$ (in the range) represent red and blue, respectively.  Then, we let $f_{[0, 1]}(c)$ be the fraction of arithmetic progressions in the coloring $c$ that are monochromatic.  We  can represent this fraction geometrically by a \emph{BCG diagram}, described by Butler, Costello, and Graham in \cite{Butler:2010} and illustrated in Figure \ref{fig:GrahamDiagram} below.  When $c$ consists of 12 contiguous segments, we label the endpoints of the coloring as $(x_0 = 0, x_1, \ldots, x_{12} = 1)$.  As we allow the coloring $c$ to vary, $f_{[0,1]}(c)$ is a piecewise quadratic function in the $x_i$.  Moreover, each piece of $f_{[0, 1]}(c)$ is determined completely by the relative ordering of the pairs of sums $\{x_i + x_j\}$, as described in Lemma \ref{lem:ShapeLemma}.

Next, we aim to identify every piece of the quadratic function over all colorings $c$ of $[0, 1]$ with $12$ intervals.  Using the GNU Linear Programming Kit \cite{GLPK}, we count $371,219$ possible arrangements of $\{x_i + x_j\}$ that could give distinct quadratics in $f_{[0, 1]}$, as proved in Lemma \ref{lem:NumOrderings} with the help of our code available online at \url{https://cocalc.com/TorinGreenwood/MonochromeSequences/MonochromaticProgressions}.

Finally, once we have identified the $371,219$ possible pieces in the quadratic function, we search for the global minimum of $f_{[0, 1]}$ among all these pieces.  Fortunately, from Lemma \ref{lem:fIsC1}, it turns out that $f_{[0, 1]}$ is a continuous function with continuous partial derivatives.  Thus, we can minimize $f_{[0, 1]}$ by searching for all critical points within each piece of the quadratic.  Because $f_{[0, 1]}$ is piecewise quadratic, its critical points are determined by systems of linear inequalities (defining the domain of a piece of $f_{[0, 1]}$) and equalities (setting the partial derivatives of $f_{[0, 1]}$ to zero), allowing us again to use linear programming to identify the critical points.  We describe our search for these critical points in Lemma \ref{lem:fMin}, completing the proof.
\end{proof}

A byproduct of our proof structure is that among colorings with a fixed number of contiguous blocks, there exist optimal colorings with rational endpoints, as described in Corollary \ref{cor:RatEndpoints}.  In Section \ref{sec:circColor}, we show that with respect to $2$-colorings of the continuous unit circle $S_1$, the fraction of monochromatic APs depends only on the measure of points colored red.  This is analogous to the results in \cite{Datskovsky:2003, Lu:2012} that concern colorings of $\mathbb{Z}_p$ for $p$ prime.

\section{Background}

When searching for bounds on the number of monochromatic arithmetic progressions in $[\intLength]$, Frankl, Graham, and R\"{o}dl developed the following theorem:

\begin{thm}[Frankl, Graham, R\"{o}dl, \cite{Frankl:1988}]
For fixed $r$ and $\progLength$, there exists $\ell > 0$ so that the number of monochromatic $\progLength$-APs in any $r$-coloring of $\{1,2,\ldots,\intLength\}$ is at least $\ell \intLength^2 + o(\intLength^2)$.
\end{thm}

This proved that a positive fraction of APs must be monochromatic in the long run, but gave no indication of how small $\ell$ could be.

Datskovsky made progress on a related problem in \cite{Datskovsky:2003}, analyzing the minimal number of monochromatic Schur triples in $[\intLength]$.  A Schur triple $(a, b, c)$ from $[\intLength]$ is any triple of integers where $a + b = c$.  Datskovsky investigated the minimum possible number of monochromatic Schur triples when coloring each integer red or blue, and proved that asymptotically, the minimum is $\intLength^2/11$.  The proof relied on using a discrete Fourier transform, which yielded a combinatorial identity that broke down counts of Schur triples into a few easier to analyze sets.  Although our proof does not use the discrete Fourier transform, it also will transform a discrete problem into a continuous space.

In \cite{Parrilo:2008}, Parrilo et al.\ applied some of the tools from Datskovsky's work to arithmetic progressions.  Again, the authors found a combinatorial identity breaking down sets of arithmetic progressions into simpler sets, but it was no longer possible to enumerate these sets exactly.  Instead, the authors ended up with bounds on the minimum number of monochromatic progressions possible in $[\intLength]$.  They also identified the coloring shown in Equation \eqref{eq:ParriloColoring} in the introduction above, and verified it was locally optimal among colorings with $12$ intervals that are \emph{antisymmetric}.  Our paper aims to prove that this coloring is optimal globally among the same set of colorings.

\emph{Constellations} are a generalization of APs studied in \cite{Butler:2010}, where instead of all points being equally spaced like in an AP, the consecutive differences of terms must satisfy some fixed proportions.  Butler et al. analyzed constellations by representing sets of monochromatic constellations using integrals of indicator functions.  This led them to represent monochromatic regions in two-dimensional diagrams which we refer to as BCG diagrams, as illustrated in Figure \ref{fig:GrahamDiagram}.  Visualizing progressions via these diagrams is crucial to our proof, and provides the connection we need between discrete and continuous realms.

One important aspect of our proof is enumerating the number of ways pairwise sums $\{x_i + x_j\}$ can be ordered for a list of positive real numbers $x_0 \leq x_1 \leq \ldots \leq x_\numBlocks$ with $\numBlocks$ even and $x_i + x_{\numBlocks - i} = 1$. This problem could be framed as counting the number of chambers in a hyperplane arrangement, and there already exists a rich set of tools for counting chambers, as seen for example in \cite{Stanley:2006}.  However, in this paper, we use mixed integer linear programming, which is well-suited to determining whether a system of linear inequalities has a solution.  This coding approach was also employed by Miller and Peterson in \cite{Miller:2019} when they counted \emph{more sums than differences sets}, and also by Laaksonen in \cite{Laaksonen:2019} when he counted closely-related arrangements of sums of pairs.  More details on this approach are given in Section \ref{sec:CountConfigurations} below.

The current best known bounds on the minimum number of monochromatic $\progLength$-APs in the general (non-antisymmetric) case for $\progLength > 3$ are found using an ``unrolling'' strategy, described in \cite{Lu:2012} and \cite{Butler:2014}.  Here, an optimal coloring of some interval $\{1, \ldots, \ell\}$ for $\ell \ll \intLength$ is found explicitly, and then repeated to fill the interval $[\intLength]$.  Although this strategy works well for $\progLength > 3$, when $\progLength = 3$, the colorings do no better than random in the long run.

\section{Relationship between discrete and continuous case}\label{sec:discCont}
%
%
%
In this section, we define a precise connection between discrete $2$-colorings of $[\intLength]$, and a natural continuous analogue of $2$-coloring $[0, 1]$. {First, we pause to define a $3$-AP in $[\intLength]$ formally: a $3$-AP is any set of $3$ terms $(a, a + d, a + 2d)$ each in $[\intLength]$ where $d$ is any integer including negative values or zero.  It is convenient for us to include the case where $d \leq 0$ in our arguments, although this choice ultimately does not change which colorings minimize monochromatic APs nor the minimum they attain.}

{For the interval $[0, 1]$, we identify any $3$-AP $(a, a + d, a + 2d)$ by its first and last term $(a, a + 2d)$ in $[0,1]\times[0,1]$, now allowing $d$ to be any real number.}  We obtain a measure on the set of $3$-APs in $[0, 1]$ by choosing the starting and ending point of the progressions uniformly.  A coloring of the interval is defined to be a function $c: [0, 1] \to \{0, 1\}$.

{In this section, we begin by discussing measurable colorings of $[0, 1]$, which can be approximated in a standard way by \emph{bead colorings}, defined below.  Then, we show that minimizing monochromatic APs over all measurable colorings of $[0, 1]$ is the same as minimizing all APs over just bead colorings, as formalized in Lemmas \ref{lem:continuity}, \ref{lem:blockapprox}, and \ref{lem:beadmeasurable} below.

Next, we justify that every discrete coloring of $[\intLength]$ has a corresponding continuous coloring of $[0, 1]$, and that the fraction of monochromatic APs in the continuous coloring is a function of both the monochromatic APs and monochromatic \emph{off-by-1} APs in the discrete coloring, as explained above and in Lemma \ref{lem:beadcompute}.  Using this connection, we find that when the number of blocks of contiguous runs of colors in a coloring is bounded by $\numBlocks$, the fraction of APs in a continuous coloring versus its discrete analogue is small as $\intLength$ grows large, formalized in Lemma \ref{cor:UniformError}.  Finally, this allows us to prove our main result of the section: that minimizing over discrete colorings with a fixed number of blocks is the same as minimizing over continuous colorings with the same number of blocks, stated rigorously in Lemma \ref{cor:DiscToContinuousFinal}.

Now, we begin stating our results formally, starting with the definition of a Lebesgue-measurable coloring.}

 \begin{definition}
     A coloring of $[0,1]$ is Lebesgue-measurable if $c^{-1}(0)$ is Lebesgue-measurable (or equivalently $c^{-1}(1)$ is Lebesgue-measurable).
 \end{definition}
 
 \begin{definition}
      A bead coloring of $[0,1]$ is a coloring where for some ${\ell}$, each of the intervals $(\frac{i}{{\ell}},\frac{i+1}{{\ell}})$ is monochrome for $i = 0, 1, \ldots, {\ell} - 1$.  Each interval $(\frac{i}{{\ell}}, \frac{i + 1}{{\ell}})$ is called a bead, and we sometimes refer to such a coloring as an ${\ell}$-bead coloring. 
 \end{definition}

 We introduce bead colorings because they are the continuous analogue of coloring the integers $[\intLength]$ obtained by fattening each integer into an interval.  Our goal is to show that when optimizing colorings over the interval $[0, 1]$, we may restrict our attention to bead colorings.  
 We call the set of bead colorings $\mathcal{B}$ and the set of Lebesgue-measurable colorings $\mathcal{M}$. Observe that $\mathcal{B} \subset \mathcal{M}$. Finally we define a difference between two colorings as follows.
 
 \begin{definition}
    For two colorings $c_a \in \mathcal{M}$ and $c_b \in \mathcal{M}$  we define $d(c_a,c_b) := \mu(\{x \mid c_a(x) \neq c_b(x)\})$, where $\mu$ is the usual Lebesgue measure on $\mathbb{R}$.
\end{definition}
 
Recall that we identify an arithmetic progression in $[0, 1]$ by the pair of starting and ending points in $[0, 1]$. 
%
For a coloring $c$ on $[\intLength]$, we define $f_{[\intLength]}(c)$ to be the fraction of arithmetic progressions that are monochromatic.
Analogously, when $c$ is a coloring of $[0, 1]$, we have the following definition:

\begin{definition}  For a coloring $c:[0, 1] \to \{0, 1\}$, let $f_{[0, 1]}(c)$ be the Lebesgue measure of the set of monochromatic arithmetic $3$-term progressions (viewed as a subset of $[0, 1]^2$) induced by the coloring $c$.
\end{definition}
We justify our restriction to bead colorings with the following standard measure-theoretic lemmas (proved for completeness momentarily):

\begin{lem}
\label{lem:continuity}
For any two measurable colorings $c_1$ and $c_2$ of $[0, 1]$, if $d(c_1,c_2) < \epsilon$,  then $|f_{[0, 1]}(c_1) - f_{[0, 1]}(c_2)| < 4 \epsilon$.
\end{lem}

\begin{lem}
\label{lem:blockapprox}
 For any measurable coloring $c_m$ of $[0, 1]$ and any $\epsilon > 0$ there exists a bead coloring $c_b$ such that $c_m$ and $c_b$ disagree on a set of measure at most $\epsilon$.
\end{lem}

As $\mathcal{B} \subset \mathcal{M}$, Lemma \ref{lem:blockapprox} immediately implies the following: 

\begin{lem}
\label{lem:beadmeasurable} Optimizing monochromatic 3-APs over bead colorings is the same as optimizing over all measurable colorings in the following sense:
\[\inf\limits_{c_b \in \mathcal{B}} f_{[0, 1]}(c_b) = \inf\limits_{c_m \in \mathcal{M}} f_{[0, 1]}(c_m).\]
\end{lem}

\ \\
We begin with the proof of Lemma \ref{lem:continuity}.

\begin{proof}[Proof of Lemma \ref{lem:continuity}.]
{Let $A \subset [0, 1]$ be a set of measure $\epsilon$, and consider flipping the colors of all elements in $A$.}  There are three classes of monochromatic $3$-APs that could be created or destroyed: the APs where the first, middle or last element is flipped (where some APs may belong to more than one class). We consider the measure of each of these three classes. As the first and last elements of a progression are chosen uniformly, the corresponding classes have measure $\epsilon$. The middle element is the average of two uniform random variables, and so has a triangular distribution on $[0,1]$ with maximum density $2$.  Therefore the set of monochrome progressions whose middle term is in $A$ would have measure at most $2 \epsilon$.  Summing the measures of these three classes yields an upper bound for their union of $4 \epsilon$.
\end{proof}

We now justify Lemma \ref{lem:blockapprox}, whose proof is a standard measure-theoretic argument.

\begin{proof}[Proof of Lemma \ref{lem:blockapprox}.]By hypothesis, the set $X_{\rm bl} := c_{{m}}^{-1}(0)$ of blue-colored elements of $[0,1]$ is measurable with finite measure.  So, a standard result from measure theory (e.g. \cite[Theorem 12]{RoydenFitz}) establishes the existence of a finite disjoint collection of open intervals $I_1, \ldots, I_\blockApproxVar \subset [0,1]$ satisfying 
\[
\mu\left(\left(\bigcup_{i=1}^\blockApproxVar I_i\right) \setminus X_{\rm bl}\right) + \mu\left(X_{\rm bl}\setminus \bigcup_{i=1}^\blockApproxVar I_i \right) < \frac{\epsilon}{2}.
\]
Since the rationals are dense in $[0,1]$, we can perturb the $2\blockApproxVar$ endpoints of the intervals $\{I_i\}$, each by some amount less than $\frac{\epsilon}{{4}\blockApproxVar}$, to find a disjoint collection $I_1', I_2', \ldots, I_\blockApproxVar'$ of open intervals with rational endpoints.  Let $\mathcal{U}_{\rm bl}$ be the union of these intervals.  Then,  $\mathcal{U}_{\rm bl}$ and $X_{\rm bl}$ have a symmetric difference of measure at most ${\epsilon}$. It follows that the coloring $c_b$ defined by coloring each interval of $\mathcal{U}_{\rm bl}$ blue is a bead coloring for which $d(c_b, c_m) < \epsilon$.
\end{proof}

%

Call a progression an \emph{off-by-1 AP} if it is of the form $(a, a + d, a + 2d \pm 1)$. We will show that we can easily compute $f_{[0, 1]}(c_b)$ for a bead coloring $c_b$ {with $\intLength$ beads by considering the colored beads as an integer coloring of $[\intLength]$}, computing the number of 3-term APs in this sequence, and adding half of the off-by-1 APs. Recall that for a discrete coloring $c$, $m_3(c)$ is the number of monochromatic $3$-APs induced by $c$.  Let $m_3'(c)$ be the number of monochromatic off-by-1 APs.  Then, we have the following comparison between colorings of $[0, 1]$ with exactly $\intLength$ beads (of not necessarily alternating colors) and corresponding colorings of $[\intLength]$.


\begin{lem}
\label{lem:beadcompute}
Let $c_b$ be an $\intLength$-bead coloring of $[0,1]$, and let $c_b^*$ be the discrete coloring of $[\intLength]$ corresponding to $c_b$, where the number $i$ is colored blue if and only if the $i$th bead in $c_b$ is colored blue.  Then,
\[
f_{[0, 1]}(c_b) = \frac{m_3(c_b^*) + m_3'(c_b^*)/2}{\intLength^2}.
\]
\end{lem}

\begin{proof}
 Consider a randomly chosen progression in $[0, 1]$ {identified by its endpoints $(a, b)$}, and a fixed $\intLength$-bead coloring $c_b$.  We use a probabilistic proof, so we rewrite
\[
(\mu\times \mu)((a,b) \in [0, 1]^2 : (a, b) \mbox{ is monochromatic}) =: \mathbb{P}((a, b) \mbox{ monochromatic}),
\]
where $\mu\times \mu$ is the usual Lebesgue measure on $\mathbb{R}^2$. We will condition on which beads $S$ and $E$ contain $a$ and $b$.  Let $M$ be the bead containing the middle element of the progression.  Given a bead coloring of $[0,1]$, it is useful to define the distance between two beads $A$ and $B$, $d_b(A,B)$ as $0$ when $A=B$ and as one more than the number of other beads strictly between $A$ and $B$ otherwise. Note that when $d_b(S, E)$ is even, then $S, M$, and $E$ must form a $3$-AP of beads.  On the other hand, when $d_b(S, E)$ is odd, $S, M,$ and $E$ form an off-by-one progression and $M$ could be two possible beads depending on the internal positioning of $a$ and $b$ within $S$ and $E$.  Formally, {letting $\{\mathcal{B}_i\}_{i = 1}^\intLength$ be the set of beads,}
\begin{align}
    f_{[0, 1]}(c_b) &= \sum_{i, j} \mathbb{P}((a, b) \mbox{ monochromatic} | a \in \mathcal{B}_i, b \in \mathcal{B}_j) \cdot \mathbb{P}(a \in \mathcal{B}_i, b \in \mathcal{B}_j) \nonumber\\
    &= \sum_{d(\mathcal{B}_i, \mathcal{B}_j) \mbox{ \scriptsize even}} \mathbb{P}((a, b) \mbox{ monochromatic} | a \in \mathcal{B}_i, b \in \mathcal{B}_j) \cdot \mathbb{P}(a \in \mathcal{B}_i, b \in \mathcal{B}_j) \nonumber \\
    & + \sum_{d(\mathcal{B}_i, \mathcal{B}_j) \mbox{ \scriptsize odd}} \mathbb{P}((a, b) \mbox{ monochromatic} | a \in \mathcal{B}_i, b \in \mathcal{B}_j) \cdot \mathbb{P}(a \in \mathcal{B}_i, b \in \mathcal{B}_j) \label{eq:ConditionProb}
\end{align}
Now, $\mathbb{P}(a \in \mathcal{B}_i, b \in \mathcal{B}_j) = 1/\intLength^2$ for each $i$ and $j$ since $a$ and $b$ are independently and uniformly distributed among the beads.  Also, {because our coloring is fixed}, when $d_b(\mathcal{B}_i, \mathcal{B}_j)$ is even $\mathbb{P}((a, b) \mbox{ monochromatic} | a \in \mathcal{B}_i, b \in \mathcal{B}_j)$ is $0$ or $1$ depending on whether or not the beads $S, M,$ and $ E$ form a monochromatic 3-term AP.  Thus,
\begin{multline}
    \sum_{d(\mathcal{B}_i, \mathcal{B}_j) \mbox{ \scriptsize even}} \mathbb{P}((a, b) \mbox{ monochromatic} | a \in \mathcal{B}_i, b \in \mathcal{B}_j) \cdot \mathbb{P}(a \in \mathcal{B}_i, b \in \mathcal{B}_j) \\
    = m_3(c_b^*) \cdot \frac{1}{\intLength^2}. \label{eq:CondOne}
\end{multline}
When $d_b(\mathcal{B}_i, \mathcal{B}_j)$ is odd, there are two choices for $M$: $\mathcal{B}_{(i + j - 1)/2}$ or $\mathcal{B}_{(i + j + 1)/2}$.  Thus, we can condition on these two choices:
\begin{align*}
&\mathbb{P}((a, b) \mbox{ monochromatic} | a \in \mathcal{B}_i, b \in \mathcal{B}_j) \\
&\hspace{1em}=\mathbb{P}((a, b) \mbox{ mono.} | a \in \mathcal{B}_i, b \in \mathcal{B}_j, M = \mathcal{B}_{(i + j - 1)/2}) \cdot \mathbb{P}(M = \mathcal{B}_{(i + j - 1)/2}|a \in \mathcal{B}_i, b \in \mathcal{B}_j)\\
&\hspace{1em}+ \mathbb{P}((a, b) \mbox{ mono.} | a \in \mathcal{B}_i, b \in \mathcal{B}_j, M = \mathcal{B}_{(i + j + 1)/2}) \cdot \mathbb{P}(M = \mathcal{B}_{(i + j + 1)/2}|a \in \mathcal{B}_i, b \in \mathcal{B}_j)
\end{align*}
Here, $\mathbb{P}(M = \mathcal{B}_{(i + j - 1)/2}| a \in \mathcal{B}_i, b \in \mathcal{B}_j) = \mathbb{P}(M = \mathcal{B}_{(i + j + 1)/2}| a \in \mathcal{B}_i, b \in \mathcal{B}_j) = 1/2$ because $a$ and $b$ are positioned uniformly within $S$ and $E$.  Additionally, 
\[
\mathbb{P}((a, b) \mbox{ monochromatic} | a \in \mathcal{B}_i, b \in \mathcal{B}_j, M = \mathcal{B}_{(i + j - 1)/2})
\]
is $0$ or $1$ depending on whether the off-by-1 progression in $c_b^*$ is monochromatic or not.  Hence, these two terms combined simplify to
\begin{align}
    \sum_{d(\mathcal{B}_i, \mathcal{B}_j) \mbox{ \scriptsize odd}} \mathbb{P}((a, b) \mbox{ monochromatic} | a \in \mathcal{B}_i, b \in \mathcal{B}_j) \cdot \mathbb{P}(a \in \mathcal{B}_i, b \in \mathcal{B}_j) \nonumber \\
    = \frac{1}{2\intLength^2} m_3'(c_b^*). \label{eq:CondTwo}
\end{align}
Plugging in Equations \eqref{eq:CondOne} and \eqref{eq:CondTwo} into Equation \eqref{eq:ConditionProb} completes the proof.
\end{proof}

Much of the rest of this paper will deal with a particular class of colorings called ``block colorings" which we now define. Informally, they are partitions of $I$ into disjoint intervals which are alternately colored red and blue. 

\begin{definition}
    For a finite collection of endpoints $\{x_i\}$ such that $0=x_0<x_1<x_2<\cdots<x_{\numBlocks-1}<x_\numBlocks=1$,  we define the associated ``block" coloring as the coloring where the $\numBlocks$ intervals $J_i=(x_{i-1},x_{i})$ (for $i \in \{1,2,\ldots,\numBlocks\}$) are all monochrome and alternate in color. 
\end{definition}

Note that the colors assigned to the endpoints $\{x_i\}$ (or indeed to any points within a measure zero set) do not matter.  With these definitions, we can now compare the performance of discrete colorings with their continuous analogues. 
\begin{lem} \label{cor:UniformError}
Let ${\mathcal{C}(\intLength, \numBlocks)}$
be the set of {$2$-}colorings of $[\intLength]$ with at most $\numBlocks$ contiguous blocks of colors.  For any coloring $c \in {\mathcal{C}(\intLength, \numBlocks)}$, let $c_*$ be the corresponding block coloring of $[0, 1]$ where the interval $[(i-1)/\intLength, i/\intLength)$ is colored blue by $c_*$ if and only if $i \in [\intLength]$ is colored blue by $c$.  Then, 
\[
\max_{c \in {\mathcal{C}(\intLength, \numBlocks)}} \left| f_{[0, 1]}(c_*) - f_{[\intLength]}(c) \right| = O\left(\frac{\numBlocks}{\intLength} \right).
\]
Here, there exists a $C > 0$ independent of $\intLength$ and $\numBlocks$ such that  $|O(\numBlocks/\intLength)| < C\numBlocks/\intLength$ for all positive integers $\numBlocks$ and $\intLength$.
\end{lem}
\begin{proof}
Our proof will use Lemma \ref{lem:beadcompute} to rewrite $f_{[0, 1]}(c_*)$ in terms of $f_{[\intLength]}(c)$.  Before proceeding with this, we will interpret the number of off-by-1 monochromatic APs induced by $c$, $m_3'(c)$, in terms of the regular monochromatic APs, $m_3(c)$.  We claim the following:
\begin{equation} \label{eq:OffBy1ToRegular}
m_3'(c) = 2m_3(c) + O(\numBlocks \intLength).
\end{equation}
To verify this, note that each AP $(a, a + d, a + 2d)$ in $[\intLength]$ corresponds {almost} bijectively to {a pair of} off-by-1 APs by moving the first or last endpoint inwards by one: $(a + 1, a + d, a + 2d)$ or $(a, a + d, a + 2d - 1)$.  (When $\intLength$ is odd, this misses exactly two off-by-1 APs: $(1, (\intLength-1)/2, \intLength)$ and $(1, (\intLength+1)/2, \intLength)$.  When $\intLength$ is even, this is truly a bijection.)

Using this near bijection, we now compare when APs and off-by-1 APs are monochromatic.  Under this contraction action, the only time an AP $(a, a + d, a + 2d)$ is monochromatic while one of its corresponding off-by-1 APs is not monochromatic is when $a$ or $a + 2d$ is adjacent to a number of the opposite color, and the same could be said if the original AP is not monochromatic but the off-by-1 AP is.  If our coloring only has $\numBlocks$ intervals total, there are only $\numBlocks-1$ ways to {position $a$ immediately before a} color change, and similarly only $\numBlocks - 1$ ways to {position $a + 2d$ immediately after} a color change.  {Since $d$ can still be chosen freely,} there are $O(\numBlocks \intLength)$ possible off-by-1 APs that disagree with their corresponding APs on being monochromatic, verifying our claim.

Now, in the notation of Lemma \ref{lem:beadcompute}, we see that $(c_*)^* = c$.  Thus,
\begin{align*}
f_{[0, 1]}(c_*) &= \frac{m_3(c)}{\intLength^2} + \frac{m_3'(c)/2}{\intLength^2}\\
&= \frac{2m_3(c) + O(\numBlocks \intLength)}{\intLength^2}\\
&= \frac{m_3(c)}{\intLength^2/2} + O\left(\frac{\numBlocks}{\intLength}\right).
\end{align*}
The proof of the lemma will be complete if we can verify the following:
\[
\frac{m_3(c)}{\intLength^2/2} = f_{[\intLength]}(c) + O\left(\frac{1}{\intLength}\right).
\]
To see this, recall that by definition $f_{[\intLength]}(c) = m_3(c)/\AP_3(\intLength)$, so that
\begin{equation} \label{eq:APOrderMagnitude}
\frac{m_3(c)}{\intLength^2/2} - f_{[\intLength]}(c) = \frac{m_3(c)}{\AP_3(\intLength)} \cdot \frac{\AP_3(\intLength) - \intLength^2/2}{\intLength^2/2}.
\end{equation}
We have $m_3(c) \leq \AP_3(\intLength)$, so that the first fraction {on the right in Equation \eqref{eq:APOrderMagnitude}} is at most $1$.  Next, note that $\AP_3(\intLength) = \intLength^2/2 + O(\intLength)$: it is easy to compute this explicitly for when $\intLength$ is even or odd.  But, intuitively, if we pick two numbers $x$ and $y$ from $[\intLength]$ at random, there are $\intLength^2$ ways to do this, and about half the time $x - y$ is even and these correspond to the start and end of a $3$-AP.  Therefore, $\AP_3(\intLength) - \intLength^2/2 = O(\intLength)$, and plugging this into Equation \eqref{eq:APOrderMagnitude} completes the proof with
\[
\frac{m_3(c)}{\intLength^2/2} - f_{[\intLength]}(c) = O\left(\frac{1}{\intLength}\right).
\]
\end{proof}

Finally, we end this section with the result rigorously justifying our conversion between discrete and continuous colorings.
\begin{lem} \label{cor:DiscToContinuousFinal}
Let $\mathcal{S}_\numBlocks$ be the block {$2$-}colorings of $[0, 1]$ with at most $\numBlocks$ blocks, and let ${\mathcal{C}(\intLength, \numBlocks)}$ be the {$2$-}colorings of $[\intLength]$ with at most $\numBlocks$ contiguous blocks, where $\numBlocks = o(\intLength)$ as $\intLength$ approaches infinity.  Then, minimizing monochromatic APs over $\mathcal{S}_\numBlocks$ is the same as minimizing monochromatic APs over ${\mathcal{C}(\intLength, \numBlocks)}$ in the following sense:
\[
\lim_{\intLength \to \infty} \left| \inf_{c \in \mathcal{S}_\numBlocks} f_{[0, 1]}(c) - \min_{c \in {\mathcal{C}(\intLength, \numBlocks)}} f_{[\intLength]}(c) \right| = 0.
\]
\end{lem}

Here, we consider block colorings of $[0, 1]$ where the edge of a block is at a possibly irrational number.  However, as we will see later, all optimal colorings of $[0, 1]$ with a fixed number of blocks must have rational endpoints.

\begin{proof}
This is mostly a standard $\epsilon$ argument, so let $\epsilon > 0$ be given.
We aim to show for all $\intLength$ sufficiently large,
\[
\left| \inf_{c \in \mathcal{S}_\numBlocks} f_{[0, 1]}(c) - \min_{c \in {\mathcal{C}(\intLength, \numBlocks)}} f_{[\intLength]}(c) \right| \leq \epsilon.
\]
We prove this in two halves, first proving the infimum is nearly bounded above by the minimum, and then arguing the reverse.  Consider any coloring $\tilde c \in \mathcal{S}_n$.  Then, by Lemma \ref{lem:continuity}, for every $\intLength$ sufficiently large, we can find a $\numBlocks$-block coloring $\tilde c_{\intLength}$ of $[0, 1]$ with endpoints of the form $r/\intLength$ for $r$ an integer such that
\begin{equation}\label{eq:InfBound1}
    \left| f_{[0, 1]}(\tilde c) - f_{[0, 1]}(\tilde c_{\intLength}) \right| < \epsilon/4.
\end{equation}
This is true because we can round each endpoint to the nearest $1/\intLength$.  Then, we define $\tilde c_\intLength^*$ to be the coloring of $[\intLength]$ where $i$ is colored blue if and only if the $i$th block of $\tilde c_\intLength$ is blue.  Note that $\tilde c_\intLength^*$ still only has at most $\numBlocks$ blocks, and that using the notation from Lemma \ref{cor:UniformError}, $(\tilde c_\intLength^*)_* = \tilde c_\intLength$.  So, from Lemma \ref{cor:UniformError}, for $\intLength$ sufficiently large (independent of the colorings $\tilde c, \tilde c_\intLength, \tilde c^*_\intLength$),
\begin{equation}\label{eq:InfBound2}
|f_{[0, 1]}(\tilde c_{\intLength}) - f_{[\intLength]}(\tilde c_{\intLength}^*)| = O(\numBlocks/\intLength). 
\end{equation}
By choosing $\intLength$ sufficiently large (independent of the colorings $\tilde c, \tilde c_\intLength, \tilde c^*_\intLength$), Equations \eqref{eq:InfBound1} and \eqref{eq:InfBound2} imply
\[
\min_{c \in {\mathcal{C}(\intLength, \numBlocks)}} f_{[\intLength]}(c) \leq f_{[\intLength]}(\tilde c_\intLength^*) < f_{[0, 1]}(\tilde c) + \epsilon
\]
where this bound holds for all $\intLength$ sufficiently large and for all $\tilde c \in \mathcal{S}_n$. Therefore, for all $\intLength$ sufficiently large,
\[
\min_{c \in {\mathcal{C}(\intLength, \numBlocks)}} f_{[\intLength]}(c) \leq \inf_{c \in \mathcal{S}_\numBlocks} f_{[0, 1]}(c) + \epsilon.
\]

Now, we prove the reverse inequality: consider any coloring $\hat c$ of $[\intLength]$, and let $\hat c^*$ be the coloring of $[0, 1]$ induced by $\hat c$.   Again, from Lemma \ref{cor:UniformError}, for $\intLength$ sufficiently large,
\[
\left|f_{[\intLength]}(\hat c) - f_{[0, 1]}(\hat c^*) \right| < \epsilon,
\]
and since this is true for any coloring $\hat c \in {\mathcal{C}(\intLength, \numBlocks)}$, this proves that for $\intLength$ sufficiently large,
\[
\inf_{c \in \mathcal{S}_\numBlocks} f_{[0, 1]}(c) \leq \min_{c \in {\mathcal{C}(\intLength, \numBlocks)}} f_{[\intLength]}(c) + \epsilon.
\]
Combining this with the complementary inequality above completes the proof.

\end{proof}

At this point, we have justified that once bounding the number of blocks in our coloring, optimizing colorings of $[\intLength]$ is the same as optimizing colorings of $[0, 1]$.  We only make use of this result when the number of blocks $\numBlocks = 12$ because that is the conjectured global optimal number of blocks.  But, the same proof shows that switching to the continuous realm works whenever $\numBlocks = o(\intLength)$ as $\intLength \to \infty$.





\section{Proofs for the continuous case}\label{sec:contProofs}

\subsection{Colorings can be represented by BCG diagrams}
Consider any block coloring $c: [0, 1] \to \{0, 1\}$ of the interval with endpoints of the blocks given by $\{x_0, x_1, \ldots, x_\numBlocks\}$ with $x_0 = 0$ and $x_\numBlocks = 1$.  Without loss of generality, assume that the first block $(x_0, x_1)$ is colored blue, and alternate colors for each remaining interval.  Recall that the colors of the endpoints of the blocks can be assigned in any way, since this does not change the probability of selecting a monochromatic progression.

\begin{figure}[t]
    \centering
    \includegraphics[width=0.7\textwidth]{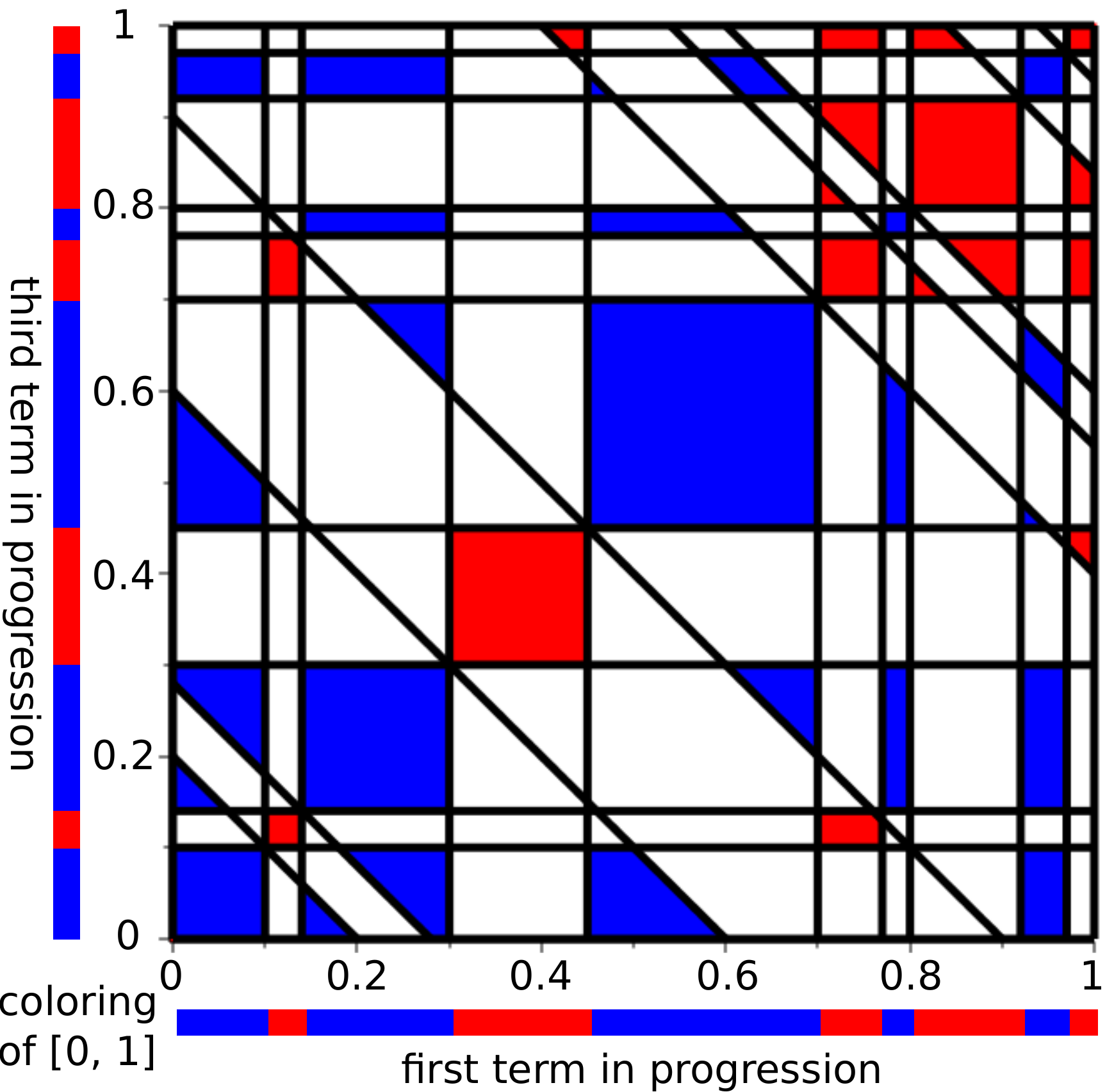}
    %
    \caption{Below the horizontal axis a coloring, $c$, is depicted.  The horizontal axis represents the first term $a$ in an arithmetic progression, and the vertical axis represents the third term $a + 2d$ in the progression.  Whenever a point in the diagram is colored red (blue), this corresponds to the progression $(a, a + d, a + 2d)$ being colored red (blue) by $c$.\\\label{fig:GrahamDiagram}}    
    \hrule
\end{figure}

In \cite{Butler:2010}, Butler, Costello, and Graham proposed a method of visualizing the monochromatic arithmetic progressions associated to a coloring in terms of diagrams like in Figure \ref{fig:GrahamDiagram}.  Any arithmetic progression $(a, a + d, a + 2d)$ can be identified uniquely by its first and last coordinates, which are represented by the horizontal and vertical axes of such a diagram.  Note that the diagram is divided into vertical strips, horizontal strips, and northwest/southeast diagonal strips.  Consider any region identified as the intersection of one horizontal, one vertical, and one diagonal strip.  For a block coloring, this region corresponds to a collection of monochromatic arithmetic progressions if and only if the indices of the vertical, horizontal, and diagonal strips defining the region all have matching parities.

Because the total area of the square in any diagram like Figure \ref{fig:GrahamDiagram} is one, the measure of the set of monochromatic sequences is equal to the sum of the areas of the red and blue regions.  In Theorem 2.1 of \cite{Butler:2010}, Butler et al.\ express the total colored area as the sum of two integrals involving an indicator function.  Their work applied to \emph{constellations}, a generalization of arithmetic progressions.  Here, we instead derive explicit polynomial equations for the areas.  Consider any one colored region in such a diagram.  As the endpoints $x_i$ are perturbed slightly, the region remains the same type of polygon although its dimensions may change.  This implies that the area of each region can be represented locally as a quadratic in the variables $\{x_i\}$.  Denote a block coloring $c$ by its list of endpoints $\mathbf{x} := (x_0, \ldots, x_\numBlocks)$.  Then, summing over all monochromatic regions shows that the measure of the monochromatic progressions, $f_{[0, 1]}(\mathbf{x})$, is locally quadratic in the $\{x_i\}$, too.  We now denote $f(\mathbf{x}) := f_{[0, 1]}(\mathbf{x})$.  
When we restrict $f$ to act on colorings with exactly $\numBlocks$ blocks, we will write $f(\mathbf{x}_\numBlocks)$.

As $\mathbf{x}$ varies, the regions in the diagram change polygon type.  Thus, for each $\numBlocks$, $f(\mathbf{x}_\numBlocks)$ is a piecewise function that is locally quadratic.  In order to minimize $f$ globally, we wish to identify the boundaries of these pieces in terms of $\mathbf{x}$.  The following lemma describes how to identify the polygons in such a diagram.

\begin{lem}\label{lem:ShapeLemma}
The region that is the intersection of the $i$th vertical strip, $j$th horizontal strip, and $k$th diagonal strip of a diagram is empty or forms a closed polygon.  The type of polygon is determined by testing whether each of the four values $\{x_{i} + x_{j}, x_{i} + x_{j+1}, x_{i+1} + x_{j}, x_{i+1} + x_{j+1}\}$ is greater than or less than the two values $\{2x_{k}, 2x_{k+1}\}$.  If this ordering is known, the area of the corresponding region can be expressed as a quadratic polynomial in the variables $\{x_{i}, x_{i+1}, x_{j}, x_{j+1}, x_{k}, x_{k+1}\}$.
\end{lem}
\begin{proof}
In the diagrams like in Figure \ref{fig:GrahamDiagram}, the horizontal lines all are given by $\{y = x_i\}_{i = 0}^\numBlocks$ and the vertical lines by $\{x = x_i\}_{i = 0}^\numBlocks$.  At any point $(x, y)$ in the diagram, the middle value in the corresponding arithmetic progression is $(x+y)/2$, and setting this equal to any endpoint in our coloring implies that the diagonal lines are given by $\{y = 2x_i - x\}_{i = 0}^\numBlocks$.  As described above, for any triple $(i, j, k)$ where $i, j, k \in \{0, \ldots, 12\}$ all have matching parities, the intersection of the $i$th vertical strip, $j$th horizontal strip, and $k$th diagonal strip corresponds to a region of monochromatic arithmetic progressions.

\begin{figure}
    \centering
    \includegraphics[width=.7\textwidth]{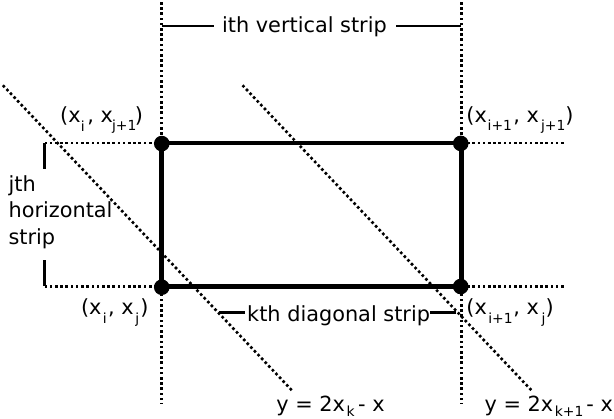}
    \caption{Above is the intersection of the $i$th vertical strip, $j$th horizontal strip, and $k$th diagonal strip determined by a block coloring with endpoints $\mathbf{x} = (x_0, x_1, \ldots, x_\numBlocks)$. Whether the intersection is empty can be determined by comparing the diagonal lines $\{y = 2x_{k} - x$, $y = 2x_{k+1} - x\}$ to the corners of the box $\{(x_{i}, x_{j}), (x_{i+1}, x_{j+1})\}$.}
    \label{fig:SingleRegion}
    \ \\\hrule
\end{figure}
To determine the shape of the region of the monochromatic progressions, first consider the rectangle formed by the intersection of the $i$th vertical strip and $j$th horizontal strip.  The corners of this rectangle have coordinates $(x_{i}, x_{j})$, $(x_{i+1}, x_{j})$, $(x_{i}, x_{j+1}),$ and $(x_{i+1}, x_{j+1})$, as labelled in Figure \ref{fig:SingleRegion}.  In order for the intersection of this rectangle with the $k$th diagonal strip to be non-empty, we need the upper diagonal line $y = 2x_{k+1} - x$ to be above the lower left corner of the rectangle, $(x_{i}, x_{j})$, and the lower diagonal line $y = 2x_{k} - x$ to be below the upper right corner of the rectangle, $(x_{i+1}, x_{j+1})$.  This is the same as requiring the inequalities $2x_{k+1} \geq x_{i} + x_{j}$ and $2x_{k} \leq x_{i+1} + x_{j+1}$.

\newlength{\figBoxWidthB}
\setlength{\figBoxWidthB}{.25\linewidth}
\newlength{\figHeightB}
\setlength{\figHeightB}{.25\linewidth}

\newcounter{FigNum}
\setcounter{FigNum}{1}

\newcommand{\polyRegBoxB}[3]{%
    \begin{minipage}{\linewidth}
    \framebox[1.02\width]{\begin{minipage}[c]{\figBoxWidthB}
    \vspace{0pt}
    \includegraphics[height = \figHeightB]{#1}
    \end{minipage}
    \begin{minipage}[c]{.72\linewidth}
    \vspace{0pt}
    $\begin{aligned}
    &\text{\textbf{Characterizing Inequalities:}}\\
    &\quad #2\\
    &\text{\textbf{Region Area:}}\\
    &\quad #3
    \end{aligned}$
    \end{minipage}}
\end{minipage}\vskip6pt
}

\begin{figure}
    \centering

\polyRegBoxB{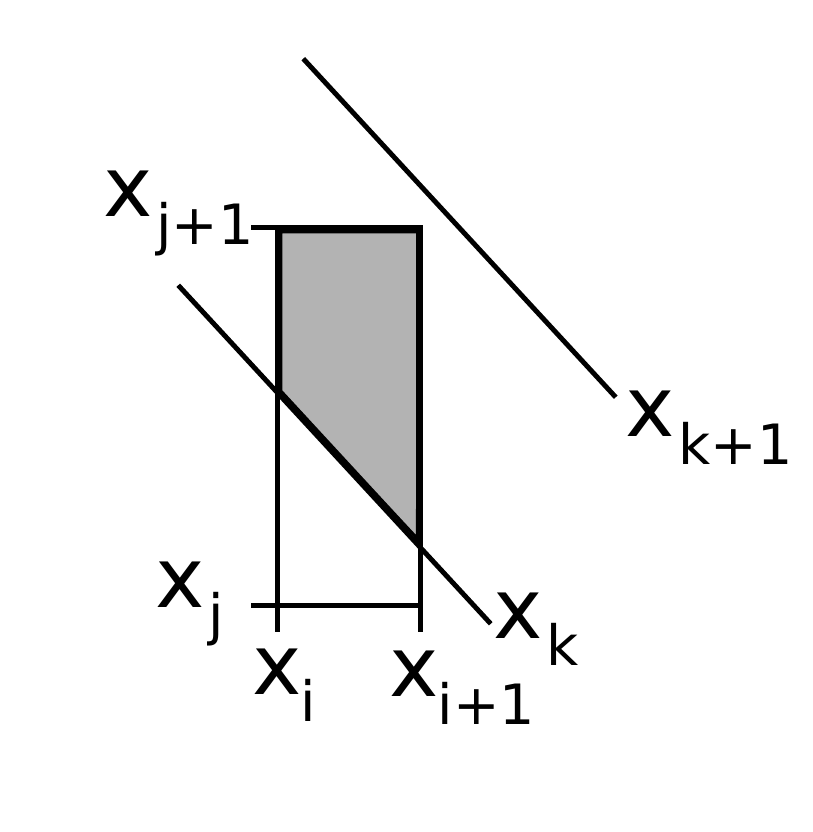}{x_{i+1} + x_{j} \leq 2x_{k} \leq x_{i} + x_{j+1} \leq x_{i+1} + x_{j+1} \leq 2x_{k+1}}{(x_{i+1} - x_{i}) ( x_{j+1} + x_{i}/2 + x_{i+1}/2  - 2x_{k})}
\polyRegBoxB{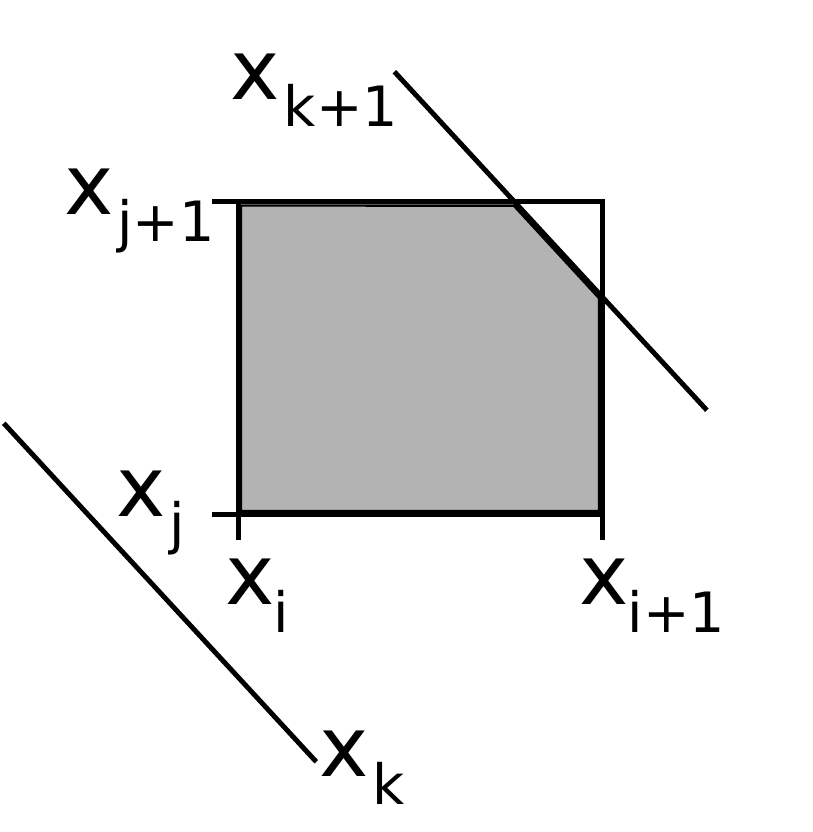}{2x_{k} \leq x_{i} + x_{j},\\ &\quad  \max(x_{i} + x_{j+1}, x_{i+1} + x_{j}) \leq 2x_{k+1} \leq x_{i+1} + x_{j+1}}{(x_{i+1} - x_{i})(x_{j+1} - x_{j}) - (2x_{k+1} - x_{j+1} - x_{i+1})^2/2}
    
    \caption{Illustrated here are two different ways that the $k$th diagonal strip can intersect with the $i$th horizontal and $j$th vertical strip in a coloring.  The resulting monochromatic region is shaded in gray, and the area of the region is given as a quadratic in $\mathbf{x}$ below the diagram.  The type of polygon is determined by the partial permutation given below each diagram.  The other 18 possibilities are enumerated in Appendix \ref{sec:AllRegions}.}
    \label{fig:SampleRegions}
    \ \\\hrule
\end{figure}

Additionally, the type of polygon formed by the intersection of the strips is determined by whether the two diagonal lines $y = 2x_{k} - x$ and $y = 2x_{k+1} - x$ are above or below each of the four corners of the box.  For any specific relationship between the lines and the four corners, some basic geometric arguments allow us to find the area of the polygon enclosed by the strips in terms of $\{x_{i}, x_{i+1}, x_{j}, x_{j+1}, x_{k}, x_{k+1}\}$.  It turns out that there are 20 possible arrangements of the lines that yield distinct polygons.  In Figure \ref{fig:SampleRegions}, two possibilities are given, along with the corresponding quadratic equations for their areas.  The full list of 20 polygons is given in Appendix \ref{sec:AllRegions}.
\end{proof}

\subsection{Enumerating BCG diagrams} \label{sec:CountConfigurations}
Now that we have identified criteria that allow us to determine the shape of each monochromatic region in a diagram, we wish to compute how many collections of shapes are possible between all diagrams.  In other words, we now know that for each fixed $\numBlocks$ the function $f(\mathbf{x}_\numBlocks)$ is a piecewise quadratic function in the endpoints $\mathbf{x}_\numBlocks$, but we would like to identify how many pieces it has.  From Lemma \ref{lem:ShapeLemma}, we have that the ordering of the pairwise sums $\{x_i + x_j\}_{0 \leq i < j \leq \numBlocks}$ completely determines the shapes in the diagram.  This is sequence A237749 in the On-Line Encyclopedia of Integer Sequences.  Currently only $9$ elements in the sequence are known, ending with $771,505,180$ possible orderings for the pairwise sums with $\numBlocks = 8$.  Thus, this sequence grows much too quickly to be useful in checking every piece of $f(\mathbf{x}_\numBlocks)$ for $\numBlocks = 12$.

Note that if $f(\mathbf{x}_\numBlocks)$ were everywhere concave up, it could only have a single local minimum, which would necessarily be the global minimum as well.  Since the conjectured optimum solution is a local minimum, the proof would be complete for any coloring with a finite number of intervals regardless of whether the coloring is antisymmetric.  Additionally, a gradient descent algorithm would quickly lead to the global minimum even if it were unknown in advance.  Unfortunately, through computational search, it is easy to find pieces of $f(\mathbf{x}_\numBlocks)$ that are not concave up.  For this reason, we must search for local minima on each piece of $f(\mathbf{x}_\numBlocks)$ individually in order to guarantee the conjectured coloring is globally minimal.

Counting pairwise orderings of $\{x_i + x_j\}_{0 \leq i < j \leq \numBlocks}$ is closely related to other combinatorial problems.  In Figure \ref{fig:CombinatorialConnections}, we find that any ordering of $\{x_i + x_j\}$ could be encoded within a standard Young tableau of inverse staircase shape.  Here, the filling of the $(r, s)$ entry of the tableau (where the top of the tableau is the $(0, 0)$ entry) is equal to the position of $x_r + x_s$ when all pairs $\{x_i + x_j\}$ are placed in increasing order.

\begin{figure}
    \centering
    \includegraphics[width=.5\textwidth]{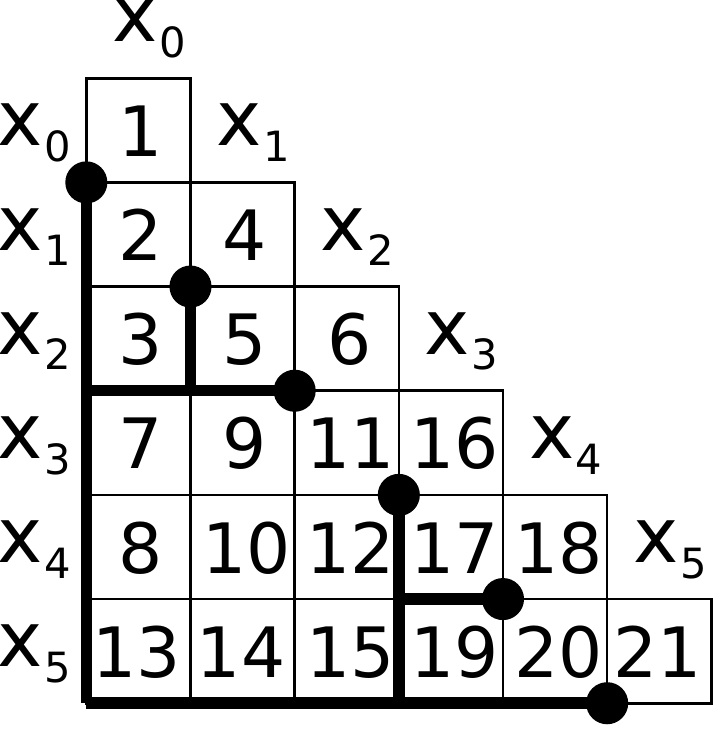}
    \caption{Above is pictured a standard Young diagram corresponding to the choice of endpoints $x_0 = 0, x_1 = 0.19, x_2 = 0.9, x_3 = 0.6, x_4 = 0.65$, and $x_5 = 1$.  The entry labelled $3$ tells us that $x_0 + x_2$ is the third smallest in the ordering of pairs $\{x_i + x_j\}_{i, j = 0}^5$.  Additionally, the non-crossing lattice paths from the line $y = -x$ to the lower left of the Young diagram partition the diagram into regions where all the corresponding pairwise sums are between two consecutive values $2x_{k-1}$ and $2x_{k}$ for some $k$.}
    \label{fig:CombinatorialConnections}
    \ \\\hrule
\end{figure}

To simplify our computations, we make two restrictions.  First, because of numerical simulations in \cite{Butler:2010} and our own, we consider only colorings that are \emph{antisymmetric}: when reflected about the middle of the unit interval, almost every point in the coloring is swapped to the opposite color.  Equivalently, we require that $x_k + x_{\numBlocks - k} = 1$ for $0 \leq k \leq \numBlocks$.  Additionally, we do not need to know all of the relations in the total ordering of $\{x_i + x_j\}_{0 \leq i < j \leq \numBlocks}$ in order to determine a configuration.  Instead, for each pair $(i, j)$ with $i \neq j$, we search for the value of $k$ where $2x_{k - 1} \leq x_i + x_j \leq 2x_k$.  Thus, we look for the number of ways to insert the pairwise sums $\{x_i + x_j\}_{i \neq j}$ into the line, $0 = 2x_0 \leq 2x_1 \leq \cdots \leq 2x_\numBlocks = 2$.  In Figure \ref{fig:CombinatorialConnections}, this reframing corresponds to not needing to know the entire filling of the diagram, but instead having a family of non-crossing lattice paths each starting at different points down the diagonal.  This set-up is very similar to the Lindstr\"{o}m-Gessel-Viennot Lemma counting non-intersecting lattice paths, \cite{Lindstrom:1973, Gessel:1985}, which was instrumental in proving the conjecture counting the number of $n \times n$ alternating sign matrices, the story of which is told in \cite{Bressoud:1999}.

To determine the number of such partial permutations, we develop an algorithm that works recursively to verify whether growing partial permutations are possible.  Our implementation is similar to Miller and Peterson's geometric approach to solving questions about \emph{More Sums Than Differences} sets, \cite[Lemma 2.1]{Miller:2019}, and also similar to Laaksonen's approach to enumerating OEIS sequence A237749, \cite{Laaksonen:2019}.  Both of these problems plus the problem we study here could be phrased in terms of enumerating chambers in a hyperplane arrangement, potentially including a restriction to a specific cone within the hyperplane arrangement.  Enumerating chambers is a stream of research on its own (e.g. \cite{Stanley:2006, Dorpalen-Berry:2021}), and there are existing theorems counting chambers by using M\"{o}bius inversion, \cite{Zaslavsky:1975, Zaslavsky:1977}.  However, here we do not need to enumerate every chamber without restrictions because this is again equivalent to counting the possible total orderings of $\{x_i + x_j\}_{0 \leq i < j \leq \numBlocks}$.

\begin{lem}\label{lem:NumOrderings}
Consider antisymmetric block colorings with endpoints $(x_0, x_1, \ldots, x_\numBlocks)$ for $\numBlocks$ even, so that $x_k + x_{\numBlocks - k} = 1$ for $0 \leq k \leq \numBlocks$.  Then, the number of ways to insert the pairs $\{x_i + x_j\}_{i \neq j}$ into the ordering $0 = 2x_0 < 2x_1 < \cdots < 2x_\numBlocks = 2$ grows as follows, starting with $\numBlocks = 0$ and with $\numBlocks$ increasing by twos:
\[
1, 1, 3, 23, 357, 9391, 371219, \ldots
\]
\end{lem}
\begin{proof}
The key computational tool in our proof is linear programming: with existing linear programming packages like the GNU Linear Programming Kit (GLPK, \cite{GLPK}) 
we can easily check whether a single system of inequalities has a valid solution.  Thus, we create a running list of partial systems of inequalities that have valid solutions, and count in how many ways it is possible to extend each system with a single additional inequality.  Below, we give pseudocode for the algorithm we use, followed by a brief explanation of some of the technicalities required to make this code run correctly and efficiently.  The full code is posted online at
\url{https://cocalc.com/TorinGreenwood/MonochromeSequences/MonochromaticProgressions}.



{%
}%
\begin{Verbatim}[numbers = left, frame = topline, framesep=9pt, vspace = 0pt, commandchars=*\{\}, label = {Pseudocode to Enumerate BCG Diagrams}]
*textcolor{blue}{\\ Initialize a running list of partial systems of inequalities}
\end{Verbatim}

\begin{Verbatim}[numbers=left, firstnumber=last, vspace = 0pt, numbersep=12pt, commandchars=\\\{\}, codes={\catcode`$=3\catcode`_=8}, frame = none]
PartialInequalitiesOld = $\big\{\{0=x_0, x_k+x_{\numBlocks-k}=1$ for $0\leq{k}\leq\numBlocks,$
                           $x_k\leq{x_{k+1}}$ for $0\leq{k}\leq\numBlocks-1\}\big\}$
                           
\end{Verbatim}

\begin{Verbatim}[numbers = left, firstnumber=last, commandchars=*\{\}, codes={\catcode`$=3\catcode`_=8}, vspace = 0pt]
*textcolor{blue}{\\ For each partial system of inequalities (i.e. for each set of}
*textcolor{blue}{partial constraints), find all ways to add a new inequality}
*textcolor{blue}{$2x_k*leq{x_i+x_j}*leq{2x_{k+1}}$ by deciding where $x_i+x_j$ fits between}
*textcolor{blue}{successive $2x_k$}
\end{Verbatim}

\begin{Verbatim}[numbers=left, firstnumber=last, vspace = 0pt, numbersep=12pt, commandchars=\\\{\}, codes={\catcode`$=3\catcode`_=8}, frame = bottomline]
for $(i,j)$ with $i\neq{j}$: \label{eqn:forij}
   PartialInequalitiesNew = \{\}
   for $k$ from $i$ to $j-1$:
      for constraints in PartialInequalitiesOld:
         if constraints $\cup$ $\{2x_k<x_i+x_j<2x_{k+1}\}$ is valid:
            PartialInequalitiesNew $+=$ $\big\{$constraints
                                         $\cup \{2x_k<x_i+x_j<2x_{k+1}\}\big\}$
   PartialInequalitiesOld = PartialInequalitiesNew
return PartialInequalitiesOld
\end{Verbatim}

We now discuss some important aspects of our implementation with GLPK that ensured the code ran efficiently and correctly.  The uninterested reader may skip the rest of this proof without a loss of continuity.  First, mixed integer linear programs typically search to optimize a linear objective function in the variables $\mathbf{x}$ over a region of linear inequalities written in terms of $\mathbf{x}$.  Here, our goal was simply to check whether a system of linear inequalities was \emph{feasible,} meaning that a solution exists.  This can be achieved with linear programming by setting the objective function to be any constant, $C$, because the linear program will return a certificate $\mathbf{x}^*$ where the maximum is achieved.  When the objective function is constant, this is simply any feasible solution.

As an added layer of complexity, linear programming typically only allows for inequalities that are not strict.  However, exponentially many arrangements of the pairs $\{x_i + x_j\}_{i \neq j}$ can be achieved trivially by the solution $\mathbf{x}^* = (0, 1/2, 1/2, \ldots, 1/2, 1)$, since any sum of distinct endpoints $x_i + x_j$ would equal $1/2, 1,$ or $3/2$.  In fact, many such arrangements can only be achieved by these trivial solutions.  If we allow such solutions, it is not possible for the program to finish due to an explosion in the number of possible systems of inequalities.  To avoid this scenario, we force all inequalities in every system to be strict.  Thus, we introduce a single auxiliary variable $\epsilon$ that converts strict inequalities into weak inequalities.  For example, the strict inequalities $2x_k < x_i + x_j < 2x_{k + 1}$ become a pair of weak inequalities $2x_k + \epsilon \leq x_i + x_j$ and $x_i + x_j + \epsilon \leq 2x_{k + 1}$.  After adding $\epsilon$ to every inequality, we change the objective function from a constant $C$ to the variable $\epsilon$, and search for the maximum value of $\epsilon$ within the region where the inequality system is true.  As long as a value of $\epsilon > 0$ is found, the set of inequalities is feasible.

Generally, linear programming implementations work with floating point arithmetic, leading to rounding errors.  Because there is no way to bound how small a feasible region could be, we used the version of GLPK that works using rational arithmetic.  Even still, GLPK returns its solutions as floating point numbers, occasionally with roundoff errors.  Thus, we set the threshold for $\epsilon$ to be near the limits of floating point arithmetic at $5 \times 10^{-15}$.  We found that the smallest $\epsilon$ value above this threshold was on the order of $10^{-3}$, illustrating that any value below $5 \times 10^{-15}$ was due to precision error.

Unfortunately, rational solvers tend to be much slower than their floating point counterparts.  To address this, we needed to optimize our code.  One factor that impacted runtime significantly was the order in which pairs $(i, j)$ were checked in the {\tt for} loop in Line \ref{eqn:forij} of the pseudocode above.  After experimenting with different orderings, we found that checking the pairs in decreasing order of $j - i$ was several times faster than checking the pairs in lexicographic order.

Additionally, the rational solver became stuck in an infinite loop for 26 of the millions of feasibility checks it ran on systems of inequalities for the $\numBlocks = 12$ case.  This issue was resolved by changing the order of the inequalities within these problematic systems of inequalities before they were input into GLPK.  We did not find a single ordering that avoided infinite loops for all of the feasibility checks.  Instead, we found that for any specific set of inequalities, there always existed some ordering where GLPK would halt rapidly.
\end{proof}

\subsection{Optimizing over all BCG diagrams}

Now that we have found the number of possible BCG diagrams for $f(\mathbf{x}_\numBlocks)$ for each $\numBlocks \leq 12$ and $\mathbf{x}_\numBlocks$ that are antisymmetric, we can finally leverage the power of calculus.  Despite being a piecewise function, we soon find that $f(\mathbf{x}_\numBlocks)$ is continuous with continuous partial derivatives.  This implies that its global maximum happens either at a critical point, or at a boundary point of the domain of the function.  In Lemma \ref{lem:fIsC1}, we prove that $f(\mathbf{x}_\numBlocks)$ has continuous partial derivatives for any fixed $\numBlocks$, after which we can finish the proof of Theorem \ref{thm:MainResult}.


\begin{lem} \label{lem:fIsC1}
Consider all block colorings with $\numBlocks$ blocks and endpoints $\mathbf{x} = (x_0, x_1, \ldots, x_\numBlocks)$.  In the region $0 = x_0 < x_1 < \ldots < x_\numBlocks = 1$, $f(\mathbf{x}_\numBlocks)$ is a continuous function with continuous partial derivatives in each variable $x_j$.
\end{lem}

\begin{proof} From the diagram representation in Figure \ref{fig:GrahamDiagram}, it is clear that $f$ is a continuous function of the endpoints, $\mathbf{x}_\numBlocks$.  To verify that the partial derivatives are continuous, we give a geometric argument: consider a single region $R$ in the diagram, like those drawn in Figure \ref{fig:SampleRegions}. Let $f_R(\mathbf{x}_\numBlocks)$ be the area of this single region as a function of the endpoints.  The region has up to $6$ sides, and each side is a line whose position is determined by some single endpoint $x_i$.  Thus, $\frac{\partial}{\partial x_i} f_R(\mathbf{x}_\numBlocks)$ is equal to the total length of the boundaries of $R$ determined by the variable $x_i$. (Indeed, moving a single $x_i$ by a small $\Delta x_i$ changes the area of the polygon $R$ by $\Delta x_i\cdot\ell_i + O(\Delta x_i)^2$ as $\Delta x_i \to 0$, where $\ell_i$ is the total length of the boundaries of $R$ determined by $x_i$.)

Now, we consider several cases.  As $\mathbf{x}_\numBlocks$ varies, $R$ may do any of the following: stay the same type of polygon, change polygon types, or enter or leave the diagram altogether.  It is clear that when $R$ stays the same type of polygon, its side lengths change continuously in $\mathbf{x}_\numBlocks$, so $f_R(\mathbf{x}_\numBlocks)$ has continuous partials in this case.  When $R$ changes polygon type, the change must occur when a diagonal line crosses over a corner of the box formed by the horizontal and vertical strips shown in Figure \ref{fig:SampleRegions}.  This means that any time a region changes polygon type, the side that enters or leaves the region does so with initial length $0$, again implying that the partials are continuous.  Finally, we consider when $R$ enters or leaves the diagram.  There are two ways this can happen: either a horizontal, vertical, or diagonal strip collapses to width $0$, or a diagonal line crosses over the corner of the box described above.  When a strip collapses to width $0$, this means that there are two consecutive endpoints $x_i$ and $x_{i + 1}$ where $(x_{i + 1} - x_i)$ tends to zero.  Thus, although the partial derivative is not continuous in this case, it is on the boundary of the region of $\mathbf{x}_\numBlocks$ values we consider.  On the other hand, when a diagonal line crosses over the corner of a box, all the side lengths of the polygon approach zero, so the partials are again continuous.

The diagram representation of $f(\mathbf{x}_\numBlocks)$ makes it clear that $f(\mathbf{x}_\numBlocks)$ has a bounded number of regions: at most one for each intersection of a horizontal, vertical, and diagonal strip.  Since $f_R(\mathbf{x}_\numBlocks)$ is continuous with continuous partial derivatives for every region $R$, $f(\mathbf{x}_\numBlocks)$ is too.
\end{proof}

Now that we have shown that $f(\mathbf{x}_\numBlocks)$ is continuous with continuous partial derivatives for a fixed $\numBlocks$, we are ready to complete the proof of Theorem \ref{thm:MainResult} with the following lemma.

\begin{lem} \label{lem:fMin}
Let $\mathbf{x}_{12} = (x_0, \ldots, x_{12})$ with $0 \leq x_0 \leq \cdots \leq x_{12} = 1$ and $\mathbf{x}_{12}$ antisymmetric.  The global minimum of $f(\mathbf{x}_{12})$ over all such $\mathbf{x}_{12}$ is $117/548$, occurring uniquely at the coloring from Equation \ref{eq:ParriloColoring}.
\end{lem}

\begin{proof}
Because $f(\mathbf{x}_{12})$ is a $C_1$ function on the polytope $0 = x_0 < x_1 < \ldots < x_{12} = 1$, its global minimum occurs on the boundary of the polytope or at a critical point within the interior of the polytope.  The boundary of this polytope is the union of polytopes of the same form with fewer variables.  For this reason, we find the critical points for $f(\mathbf{x}_\numBlocks)$ for each even value of $\numBlocks$ between $0$ and $12$.

Lemma \ref{lem:ShapeLemma} implies that $f(\mathbf{x}_\numBlocks)$ is a piecewise-quadratic function for each $\numBlocks$.  Fix $\numBlocks$, and consider any piece of this function, which can be extended to a function $f^*(\mathbf{x}_\numBlocks)$ on all of $\mathbb{R}^{\numBlocks/2 - 1}$ (since $x_1$ through $x_{\numBlocks/2 - 1}$ determine the coloring because it is anti-symmetric).  The partial derivatives of $f^*(\mathbf{x}_\numBlocks)$ are piecewise linear functions.  The critical points of this everywhere-defined quadratic are the solution to a linear system of equations.  Therefore, there are either no critical points, or a vector space of critical points.  In the case that the vector space has positive dimension, the value of $f^*(\mathbf{x}_\numBlocks)$ must be constant among all of its critical points.  Thus, when $f^*(\mathbf{x}_\numBlocks)$ has critical points, it suffices to check the value of $f^*(\mathbf{x}_\numBlocks)$ at a single critical point when checking for the values of local optima.

This leads us to the following pseudocode to search for the global minimum of $f(\mathbf{x}_{12})$ on the polytope $0 = x_0 \leq x_1 \leq \cdots \leq x_{12} = 1$.  (The full version of the code is posted at \url{https://cocalc.com/TorinGreenwood/MonochromeSequences/MonochromaticProgressions}.)\\\




\begin{Verbatim}[numbers = left,  firstnumber=last, framesep=9pt, vspace=0pt, frame = topline, commandchars=&\{\}, codes={\catcode`$=3\catcode`_=8\catcode`^=7}, label ={Pseudocode to find the minimum value of $f(\mathbf{x}_{12})$}]
&textcolor{blue}{\\ We search the interior of $f(&mathbf{x}_&numBlocks)$ for $&numBlocks=2,4,6,8,10,$ and $12$.}
>> for $&numBlocks$ from 0 to 12 by twos:

   >> for each piece $f^*(&mathbf{x}_&numBlocks)$ of the piecewise function $f(&mathbf{x}_&numBlocks)$
      (identified by Lemma &ref{lem:NumOrderings}):
      
         >> calculate the quadratic polynomial corresponding to
            $f^*(&mathbf{x}_&numBlocks)$ (by using Lemma &ref{lem:ShapeLemma})
         
         &textcolor{blue}{\\In the next line, we can feed into GLPK all of the}
         &textcolor{blue}{inequalities defining the configuration for $f^*(&mathbf{x}_&numBlocks)$ plus}
         &textcolor{blue}{the equalities that set each of the partial derivatives}
         &textcolor{blue}{of $f^*(&mathbf{x}_&numBlocks)$ to zero.}
         >> use GLPK to check the existence of a critical point
            of $f^*(&mathbf{x}_&numBlocks)$ within the region of $&mathbf{x}_&numBlocks$-values where
            $f(&mathbf{x}_&numBlocks)&equiv{f^*(&mathbf{x}_&numBlocks)}$
         
         >> if critical points exist:
            >> evaluate $f^*(&mathbf{x}_&numBlocks)$ at any critical point $&mathbf{c}_&numBlocks$
            >> store $&mathbf{c}_&numBlocks$ and $f^*(&mathbf{c}_&numBlocks)$ if this is a new record minimum
            
\end{Verbatim}
\begin{Verbatim}[numbers = left, firstnumber=last, vspace=0pt, frame = bottomline, commandchars=&\{\}, codes={\catcode`$=3\catcode`_=8\catcode`^=7}]
>> return the minimum $&mathbf{c}_&numBlocks$ and $f^*(&mathbf{c}_&numBlocks)$ values
\end{Verbatim}

This code verifies that the global minimum of $f(\mathbf{x}_\numBlocks)$ when $\numBlocks$ is at most $12$ is $117/548$, which is attained only at the coloring with endpoints given in Equation \eqref{eq:ParriloColoring} (without the $\intLength$ in each coordinate).
\end{proof}

The number of pieces of the function $f(\mathbf{x}_\numBlocks)$ for $\numBlocks \geq 14$ grows very rapidly, making an analysis of its critical points increasingly challenging.  However, we can guarantee that the optimal is always rational:

\begin{cor}\label{cor:RatEndpoints} For each $\numBlocks \in \mathbb{Z}^+$, the minimum value of $f(\mathbf{x}_\numBlocks)$ is rational, regardless of whether $\mathbf{x}_\numBlocks$ is restricted to be anti-symmetric or not.
\end{cor}
\begin{proof} 
This is nearly immediate from our proof structure: the minimum of $f(\mathbf{x}_\numBlocks)$ occurs at some critical point of $f(\mathbf{x}_\ell)$ with $\mathbf{x}_\ell$ in the interior of where $f(\mathbf{x}_\ell)$ is defined, for an $\ell \leq \numBlocks$.  These critical points are defined by a system of linear equations with rational coefficients.  Whenever there are only finitely many critical points, they all must have rational coordinates.  On the other hand, if there is a piece $f^*(\mathbf{x}_\ell)$ of the piecewise function $f(\mathbf{x}_\ell)$ that has infinitely many critical points, all of the critical points of $f^*(\mathbf{x}_\ell)$ attain the same constant value.  This implies that there still exists a critical point with rational coordinates where the minimum is attained.  Finally, because each piece of $f(\mathbf{x}_\ell)$ is a quadratic with rational coefficients, the minimum is thus also rational.
\end{proof}

\section{Circle colorings}\label{sec:circColor}
As a variation on the theme of enumerating monochromatic progressions within colorings of $[\intLength]$, some authors have also investigated properties of arithmetic progressions within colorings of the cyclic group $\mathbb{Z}_\intLength$.  For example, given a fixed red and blue $2$-coloring of $\mathbb{Z}_p$ for $p$ prime, the fraction of monochromatic $3$-term progressions that are red or blue depends only on the proportion of elements colored red, and not on the exact positioning of the red and blue elements, \cite{Datskovsky:2003, Lu:2012}. Even when $\intLength$ is not prime, the fraction of monochromatic $3$-term progressions in $\mathbb{Z}_\intLength$ is bounded below by the quantity given if $\intLength$ were prime, \cite{Lu:2012}.

Inspired by these results, we now explore a continuous analogue to the enumeration of monochromatic progressions within $2$-colorings of $\mathbb{Z}_\intLength$. Color each of the numbers in the unit circle $S_1 = \{e^{2\pi i \theta}: \theta \in [0,1) \}$ with red or blue, and consider $3$-term arithmetic progressions of the form $(e^{2\pi i x_1}$, $e^{2\pi i (x_1+d)}$, $e^{2\pi i (x_1+2d)})$ for $x_1, d \in [0,1)$. To properly discuss the ``fraction'' of these that are monochromatic for a given coloring, we introduce the uniform probability measure $\mu$ on $[0,1)$ and randomly sample arithmetic progressions by independently choosing $x_1,d \in [0,1)$ according to $\mu$. Using this framework, the probability of selecting a monochromatic progression depends only on the Lebesgue measure of the set of points colored red (i.e.\ the likelihood that, say, $e^{2\pi i x_1}$ is red) and not on which points were colored red, which is an analogous result to the one for $2$-colorings of the discrete group $\mathbb{Z}_p$.

\begin{lem} \label{Lem:CircleIndependence}
Let $C:S_1 \to \{0,1\}$ be any measurable coloring of $S_1$ with
\[
p:=\mu\left(\left\{\theta \in [0,1): C\left(e^{2\pi i \theta}\right) = 0\right\}\right)
\]
defined as the proportion of points colored red, and let $m(C)$ be the set containing all pairs $(x_1, d) \in [0,1)\times [0,1)$ such that $(e^{2\pi i x_1}$, $e^{2\pi i (x_1+d)}$, $e^{2\pi i (x_1+2d)})$ are monochromatic. Then,
\[
(\mu\times \mu)(m(C)) = 1 - 3p + 3p^2.
\]
In particular, if we randomly select a starting point $x_1$ and an increment $d$ independently from each other according to the uniform distribution on $S_1$, then the probability that the associated $3$-AP is monochrome depends only on the proportion $p$ of red points and not on how these points are distributed around $S_1$.
\end{lem}


\begin{proof}
We take a probabilistic approach that follows the proof structure of Theorem 6 from \cite{Lu:2012}.  To that end, let $x_1$ and $d$ be independent draws from $\mu$ and for $i=1,2,3$ let $A_i$ (respectively, $B_i$) be the event that the $i$th term in the progression $(e^{2\pi i x_1}$, $e^{2\pi i (x_1+d)}$, $e^{2\pi i (x_1+2d)})$ is red (respectively, blue). Then, via inclusion/exclusion, we have
\[
\mathbb{P}(A_1 \cup A_2 \cup A_3)=\left(\sum_{i = 1}^3 \mathbb{P}(A_i)\right) - \left(\sum_{1 \leq i < j \leq 3} \mathbb{P}(A_i \cap A_j)\right) + \mathbb{P}(A_1 \cap A_2 \cap A_3).
\]
We note that $\mathbb{P}(A_1 \cup A_2 \cup A_3) = 1 - \mathbb{P}(B_1 \cap B_2 \cap B_3)$.  Since $\mathbb{P}(m(C)) = \mathbb{P}(A_1 \cap A_2 \cap A_3) + \mathbb{P}(B_1 \cap B_2 \cap B_3)$, we can rearrange the above to obtain
\begin{equation}\label{eqn:circProb}
\mathbb{P}(m(C)) = 1 - \sum_{i=1}^3 \mathbb{P}(A_i) + \sum_{1 \leq i < j \leq 3} \mathbb{P}(A_i \cap A_j).
\end{equation}
The random variables $e^{2\pi i x_1}$, $e^{2\pi i (x_1+d)}$, and $e^{2\pi i (x_1+2d)}$ are pairwise independent and uniformly distributed on $S_1$. This is true based on the rotational invariance of the uniform distribution on $S_1$ and the fact that $e^{2\pi i x_1}$ and $e^{2\pi i d}$ are independent and uniformly distributed on $S_1$.  It follows that $\mathbb{P}(A_i) = p$ and $\mathbb{P}(A_i \cap A_j) = p^2$ for $i,j = 1,2,3$.  Substituting these into \eqref{eqn:circProb} yields
%
%
%
\[
(\mu \times \mu)(m(C)) = \mathbb{P}(m(C)) = 1 - 3p + 3p^2.
\]
\end{proof}

\section{Future Work}
In Sections \ref{sec:discCont} and \ref{sec:contProofs} above, we outlined an approach to identifying the optimal coloring of $[\intLength]$ and the interval $[0, 1]$ that minimizes the fraction of monochromatic $3$-APs for any fixed upper bound on the number of blocks $\numBlocks$.  A natural question is whether we can show that the optimal coloring for any $\numBlocks > 12$ is the same as the optimal coloring for $\numBlocks = 12$.  One possibility is to prove that the colorings are no better for $\numBlocks = 14$, and then argue that adding arbitrarily more intervals is no better than adding just two more intervals.

Besides investigating how colorings of $[\intLength]$, $\mathbb{Z}_\intLength$, $[0,1]$, and $S_1$ affect the prevalence of monochromatic arithmetic progressions of length 3, there are other related problems that have yet to be explored. Perhaps the most natural question to ask is how the analysis changes if we consider longer arithmetic progressions, and the articles \cite{Wolf:2010, Butler:2010, Lu:2012, Butler:2014} make  partial progress in this direction for several different lengths of progressions.  A slightly less obvious question is to ask what happens when we consider arithmetic progressions of color-dependent lengths.  For example, we could attempt to color $[0,1]$ or $[\intLength]$ in a way that simultaneously minimizes the fractions of monochromatic blue progressions of length 3 and monochromatic red progressions of length 4. 

Another natural generalization is to add more colors.  What do the $3$- and $4$-colorings of $[0,1]$ that minimize monochrome arithmetic progressions of length 3 look like?  Can anything be said about the rate at which the fraction of monochrome progressions decays as the number of colors increases? All of these questions have natural analogues in the setting of Ramsey theory as applied to graphs, and of course these generalizations might interact in any number of ways. 

When the problems studied in this paper were first posed, it was unclear whether or not colorings could perform better than random.  Although they can perform better than random in the cases we present in detail above, is this also true for related problems?  Recent work in \cite{Costello:2021} gives interesting insights into some classes of problems where solutions must be better than random.

In addition to changing the number of colors or length of the progressions we study, we could also consider colorings in other geometries.  For example, we wonder how to color an interval that has a gap in the middle in order to minimize monochromatic APs therein. By varying the length of the gap, we might gain insight into why antisymmetry is seemingly important in the optimal block colorings of $[0,1]$ that we discuss above.  Furthermore, we have already seen that in the contexts of $\mathbb{Z}_p$ for $p$ prime and the continuous circle, the performance of colorings with respect to $3$-term progressions depends only on the ratios of the colors present.  What other algebraic and geometric settings exhibit similar behavior? Alternatively, what would happen if we were to consider $S_1$ as in Section \ref{sec:circColor} but sample $3$-APs by choosing the start point and increment according to a different distribution than uniform?

\section{Acknowledgments}

Computations were performed using High Performance Computing infrastructure provided by the Mathematical Sciences Support unit at the University of the Witwatersrand, and for this the authors are thankful.
%
%
Additionally, the authors are grateful for invaluable tips from Professor Antti Laaksonen on how to optimize the code in Lemma \ref{lem:NumOrderings}.

\printbibliography

\appendix

\section{Appendix: 20 Polygonal Regions} \label{sec:AllRegions}
The 20 possible regions from Lemma \ref{lem:ShapeLemma} are given below.\\

\newlength{\figBoxWidth}
\setlength{\figBoxWidth}{.19\linewidth}
\newlength{\figHeight}
\setlength{\figHeight}{.19\linewidth}

\setcounter{FigNum}{1}

\newcommand{\polyRegBox}[3]{%
    \begin{minipage}{1.5em}\flushright\textbf{\arabic{FigNum}}\addtocounter{FigNum}{1}\end{minipage}\hspace{3pt}
    \begin{minipage}{\linewidth}
    \framebox[1.02\width]{\begin{minipage}[c]{\figBoxWidth}
    \vspace{0pt}
    \includegraphics[height = \figHeight]{#1}
    \end{minipage}
    \begin{minipage}[c]{.745\linewidth}
    \vspace{0pt}
    $\begin{aligned}
    &\text{\textbf{Characterizing Inequalities:}}\\
    &\quad #2\\
    &\text{\textbf{Region Area:}}\\
    &\quad #3
    \end{aligned}$
    \end{minipage}}
\end{minipage}\vskip6pt
}


{%
\parindent0pt
\polyRegBox{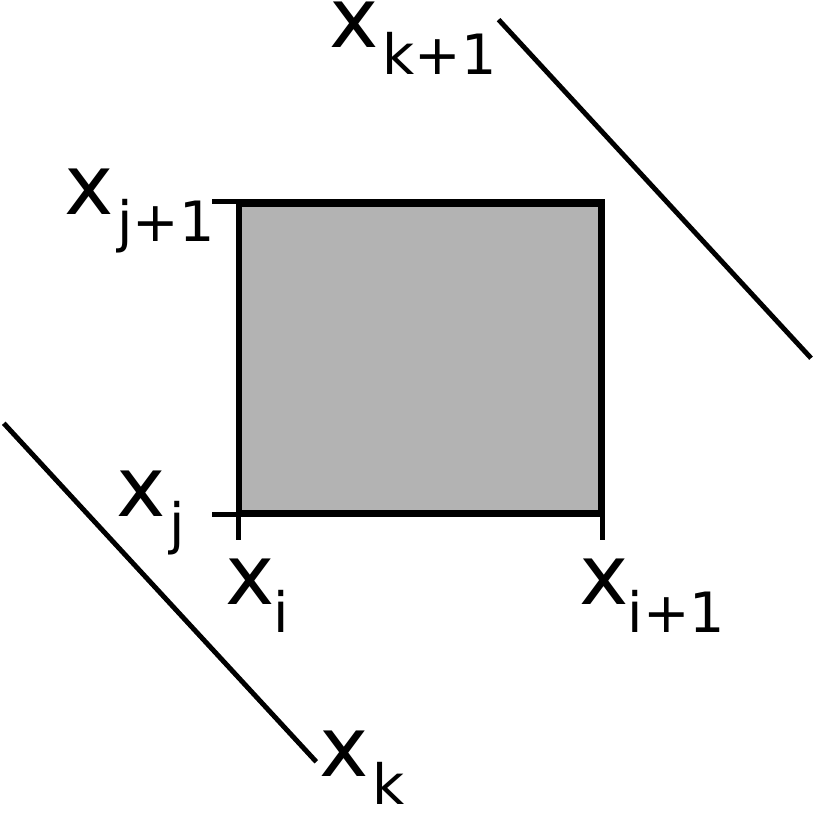}{2x_{k} \leq x_{i} + x_{j} \leq x_{i+1} + x_{j+1} \leq 2x_{k+1}}{(x_{i+1} - x_{i})(x_{j+1} - x_{j})}
\polyRegBox{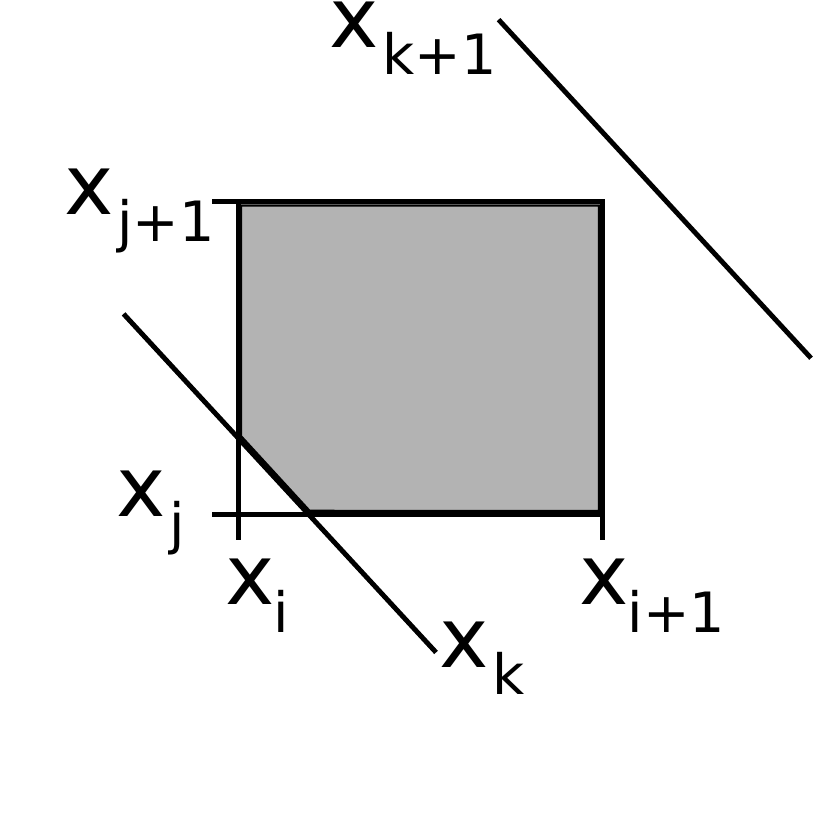}{x_{i} + x_{j} \leq 2x_{k} \leq \min(x_{i} + x_{j+1}, x_{i+1} + x_{j}),\\&\quad x_{i+1} + x_{j+1} \leq 2x_{k+1}}{(x_{i+1} - x_{i})(x_{j+1} - x_{j}) - (2x_{k} - x_{i} - x_{j})^2/2}
\polyRegBox{PermutationCases/Case3.pdf}{x_{i+1} + x_{j} \leq 2x_{k} \leq x_{i} + x_{j+1} \leq x_{i+1} + x_{j+1} \leq 2x_{k+1}}{(x_{i+1} - x_{i}) ( x_{j+1} + x_{i}/2 + x_{i+1}/2  - 2x_{k})}
\polyRegBox{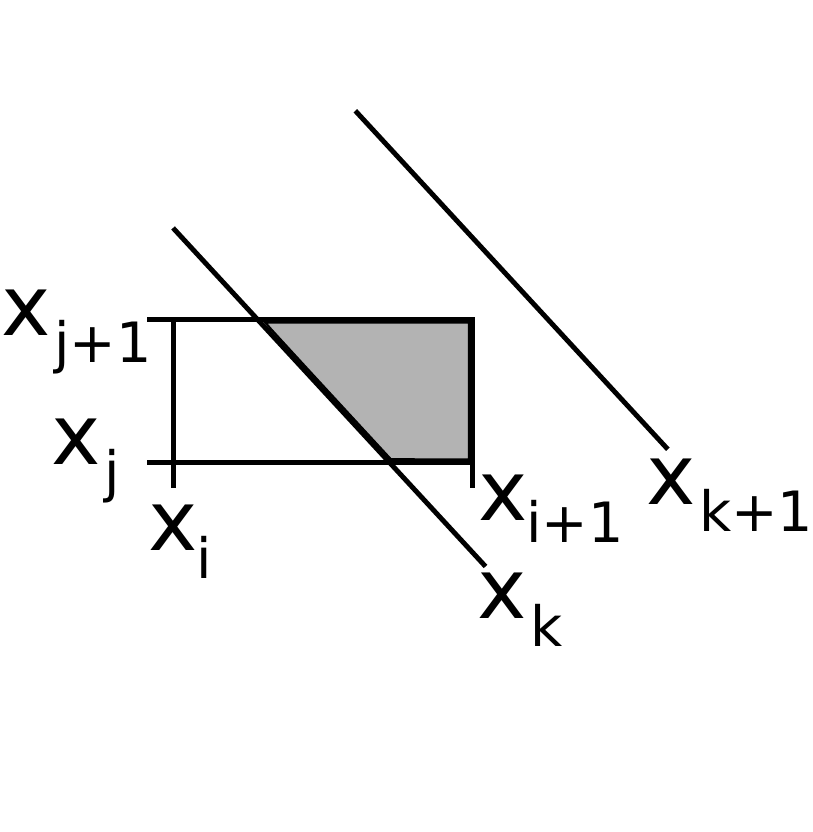}{x_{i} + x_{j+1} \leq 2x_{k} \leq x_{i+1} + x_{j} \leq x_{i+1} + x_{j+1} \leq 2x_{k+1}}{(x_{j+1} - x_{j}) (x_{i+1} + x_{j}/2 + x_{j+1}/2 - 2x_{k})}
\polyRegBox{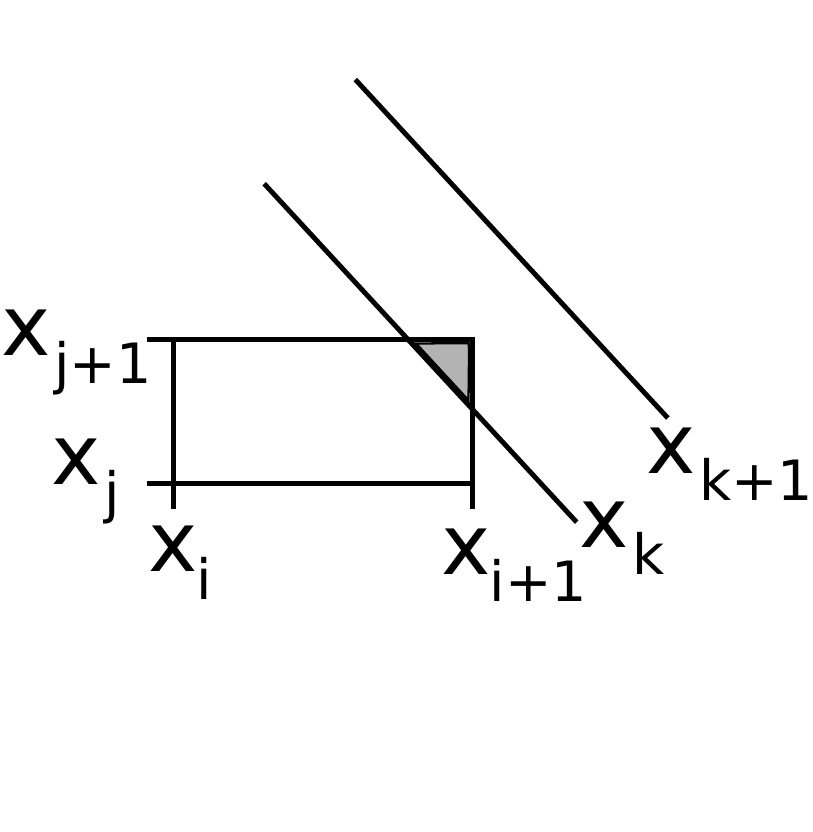}{\max(x_{i} + x_{j+1}, x_{i+1} + x_{j}) \leq 2x_{k} \leq x_{i+1} + x_{j+1} \leq 2x_{k+1}}{
(2x_{k} - x_{i+1} - x_{j+1})^2/2}
\polyRegBox{PermutationCases/Case6.pdf}{2x_{k} \leq x_{i} + x_{j},\\ &\quad  \max(x_{i} + x_{j+1}, x_{i+1} + x_{j}) \leq 2x_{k+1} \leq x_{i+1} + x_{j+1}}{(x_{i+1} - x_{i})(x_{j+1} - x_{j}) - (2x_{k+1} - x_{j+1} - x_{i+1})^2/2}
\polyRegBox{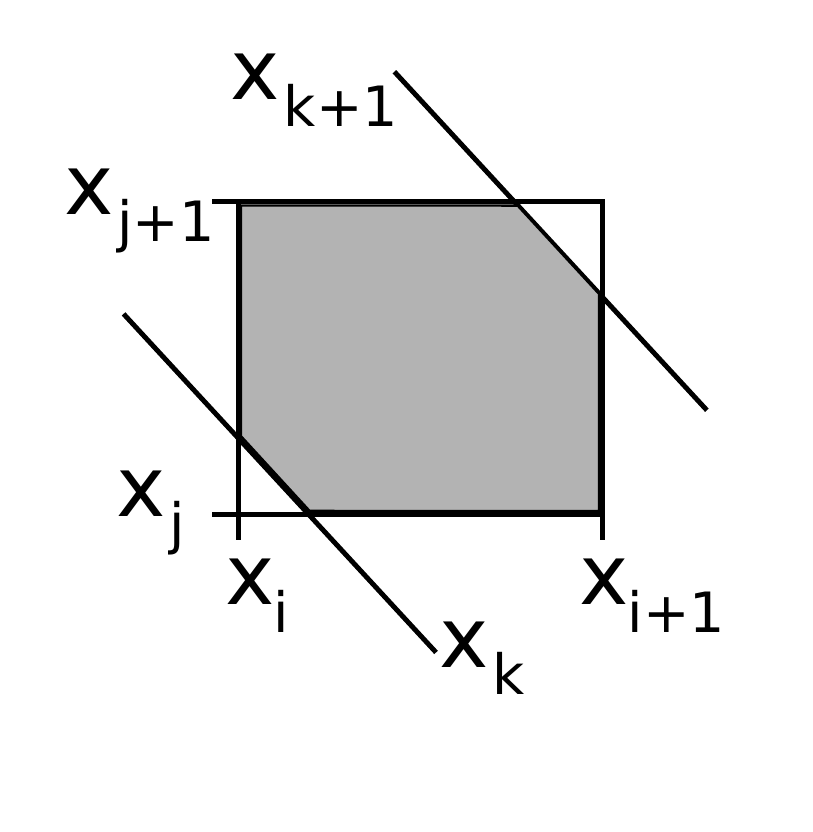}{x_{i} + x_{j} \leq 2x_{k} \leq \min(x_{i} + x_{j+1}, x_{i+1} + x_{j}),\\&\quad \max(x_{i} + x_{j+1}, x_{i+1} + x_{j}) \leq 2x_{k+1} \leq x_{i+1} + x_{j+1}}{(x_{i+1} - x_{i})(x_{j+1} - x_{j}) - (2x_{k} - x_{i} - x_{j})^2/2 \\&\quad-  (2x_{k+1} - x_{i+1} - x_{j+1})^2/2}
\polyRegBox{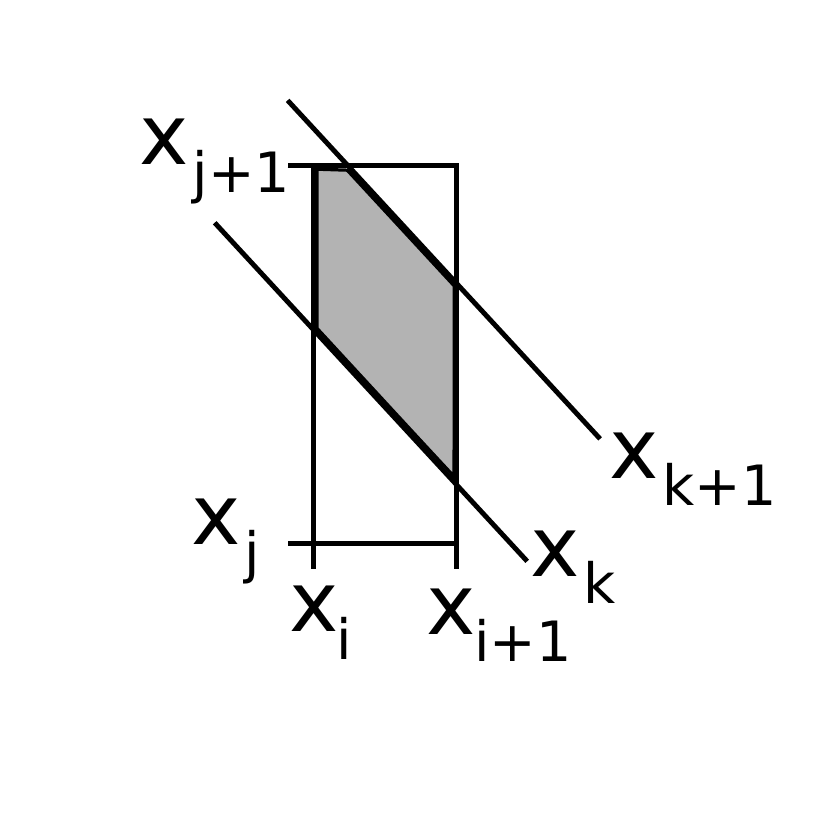}{x_{i+1} + x_{j} \leq 2x_{k} \leq x_{i} + x_{j+1} \leq 2x_{k+1} \leq x_{i+1} + x_{j+1}}{(x_{i_2} - x_{i_1}) (x_{j_2} + x_{i_1}/2 + x_{i_2}/2 - 2x_{k_1}) \\&\quad - (2x_{k_2} - x_{i_2} - x_{j_2})^2/2 }
\polyRegBox{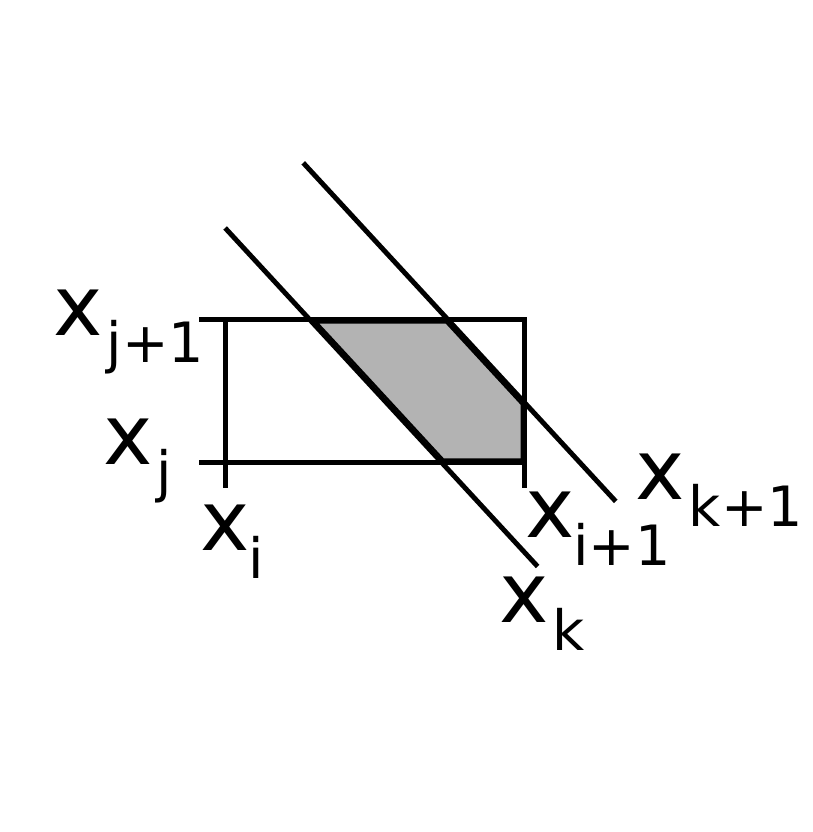}{x_{i} + x_{j+1} \leq 2x_{k} \leq x_{i+1} + x_{j} \leq 2x_{k+1} \leq x_{i+1} + x_{j+1}}{(x_{j+1} - x_{j}) \cdot(x_{i+1} + x_{j}/2 + x_{j+1}/2 - 2x_{k}) \\&\quad- (2x_{k+1} - x_{i+1} - x_{j+1})^2/2}
\polyRegBox{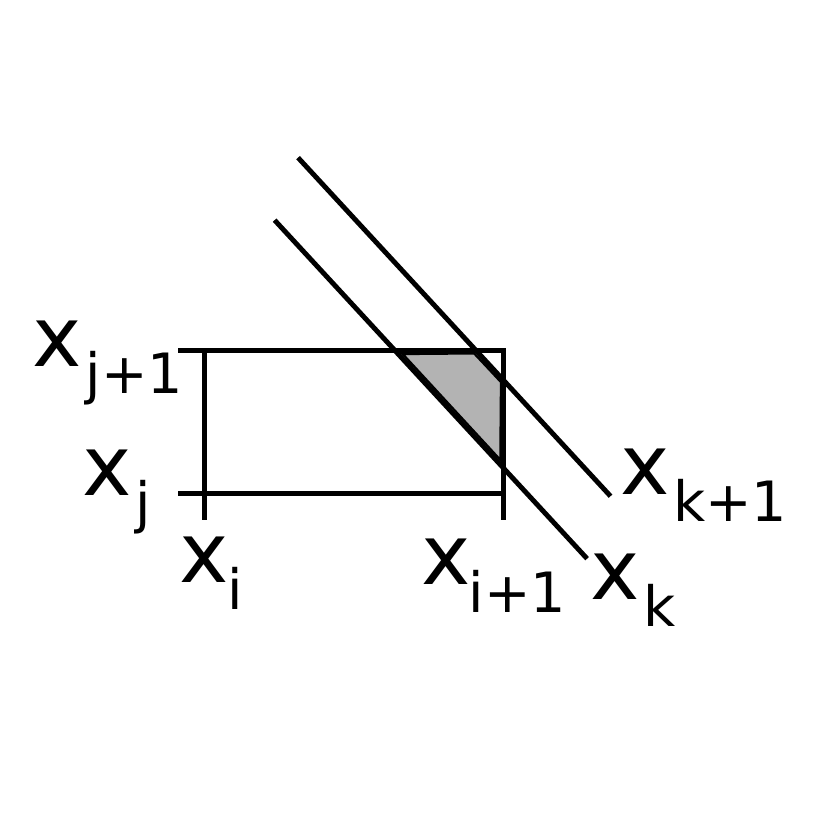}{\max(x_{i} + x_{j+1}, x_{i+1} + x_{j}) \leq 2x_{k} \leq 2x_{k+1} \leq x_{i+1} + x_{j+1}}{(2x_{k} - x_{i+1} - x_{j+1})^2/2 - (2x_{k+1} - x_{i+1} - x_{j+1})^2/2}
\polyRegBox{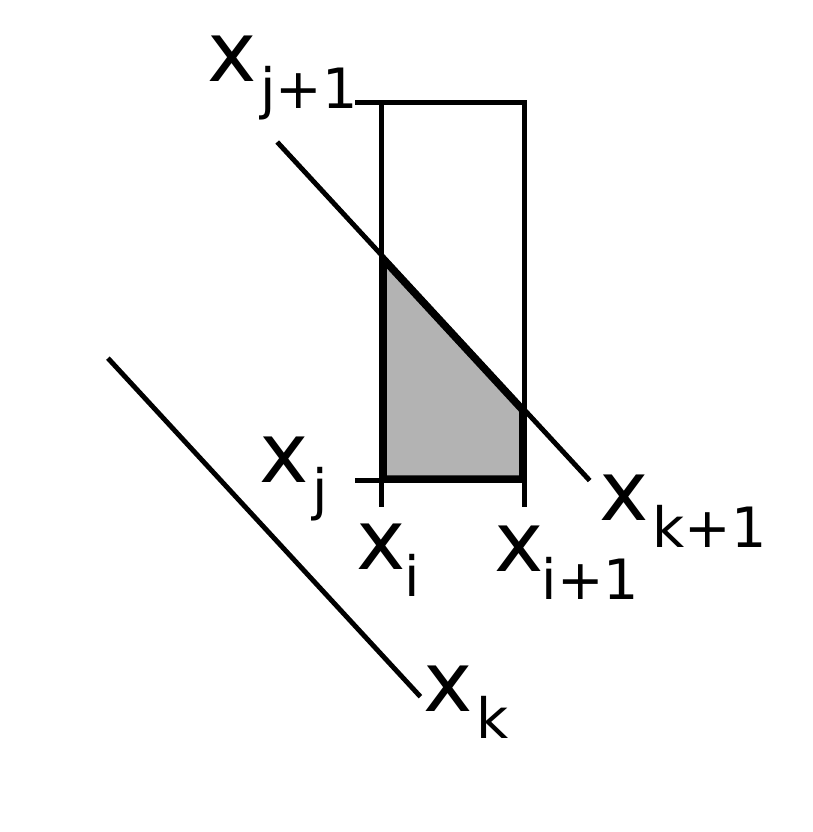}{2x_{k} \leq x_{i} + x_{j} \leq x_{i+1} + x_{j} \leq 2x_{k+1} \leq x_{i} + x_{j+1}}{(x_{i+1} - x_{i})(2x_{k+1} - x_{j} - x_{i}/2 - x_{i+1}/2)}
\polyRegBox{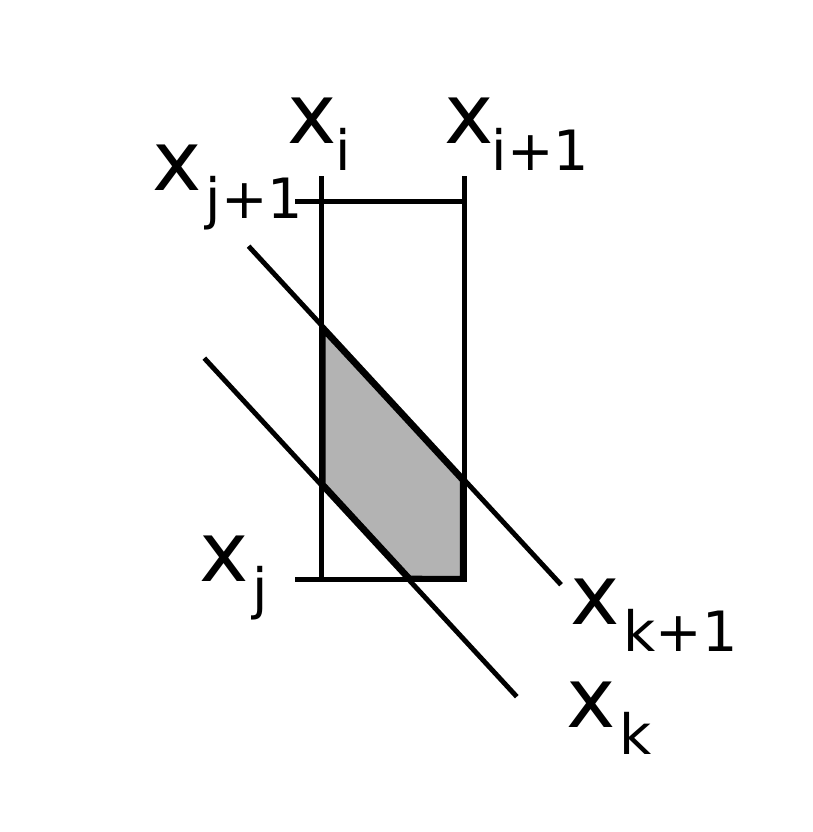}{x_{i} + x_{j} \leq 2x_{k} \leq x_{i+1} + x_{j} \leq 2x_{k+1} \leq x_{i} + x_{j+1}}{(x_{i+1} - x_{i})(2x_{k+1} - x_{j} - x_{i}/2 - x_{i+1}/2) \\&\quad- (2x_{k} - x_{i} - x_{j})^2/2}
\polyRegBox{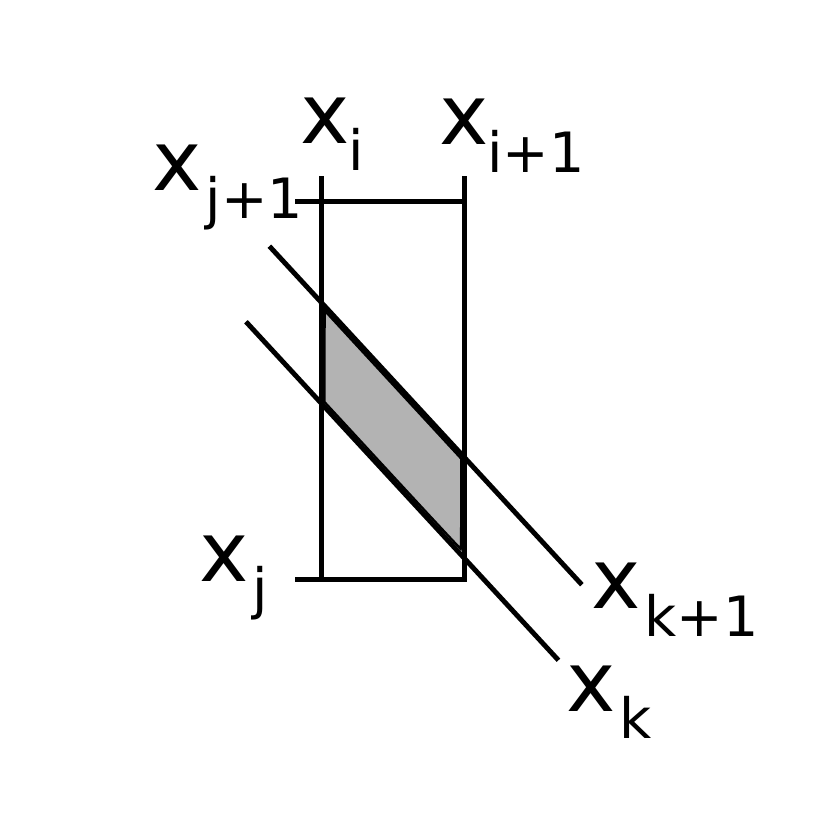}{x_{i+1} + x_{j} \leq 2x_{k} \leq 2x_{k+1} \leq x_{i} + x_{j+1}}{(x_{i+1} - x_{i})(2x_{k+1} - 2x_{k})}
\polyRegBox{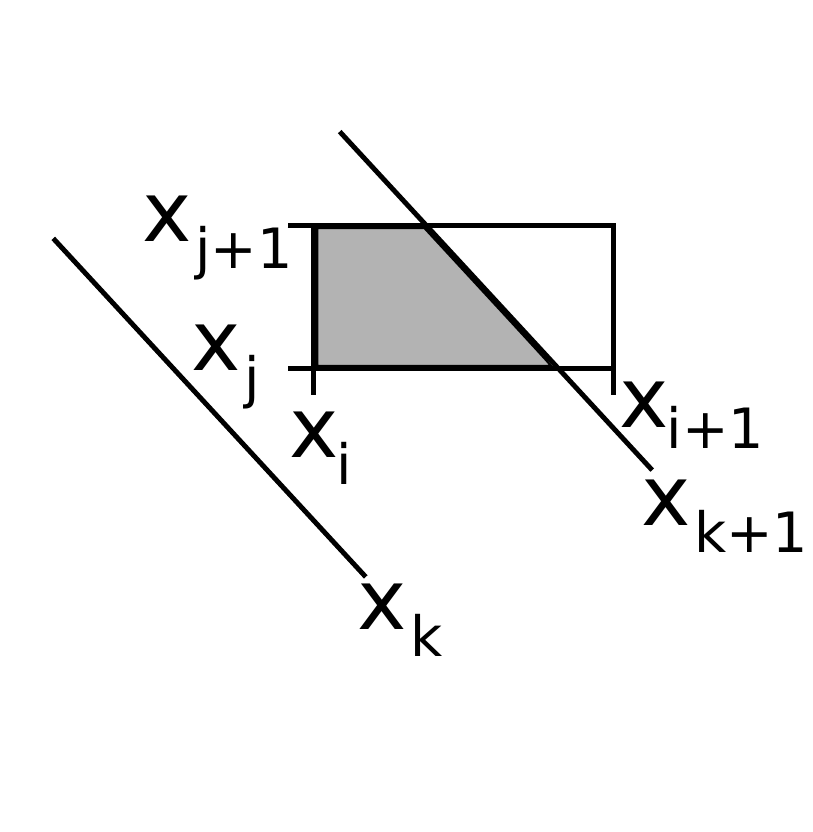}{2x_{k} \leq x_{i} + x_{j} \leq x_{i} + x_{j+1} \leq 2x_{k+1} \leq x_{i+1} + x_{j}}{(x_{j+1} - x_{j}) (2x_{k+1} - x_{j}/2 - x_{j+1}/2 - x_{i})}
\polyRegBox{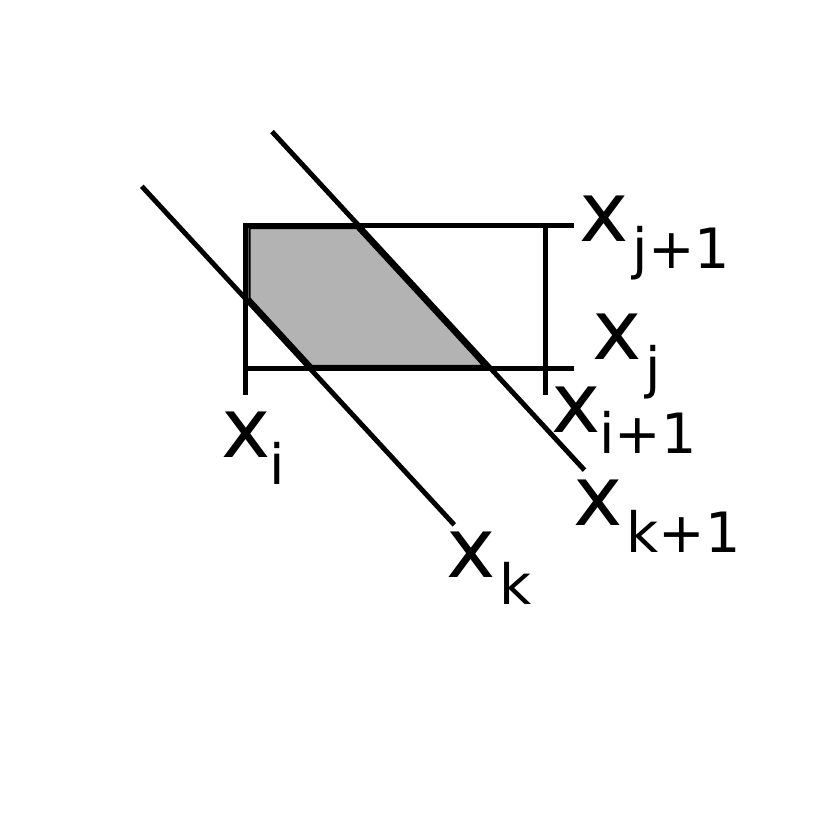}{x_{i} + x_{j} \leq 2x_{k} \leq x_{i} + x_{j+1} \leq 2x_{k+1} \leq x_{i+1} + x_{j}}{(x_{j+1} - x_{j})(2x_{k+1} - x_{j}/2 - x_{j+1}/2 - x_{i}) \\&\quad- (2x_{k} - x_{i} - x_{j})^2/2}
\polyRegBox{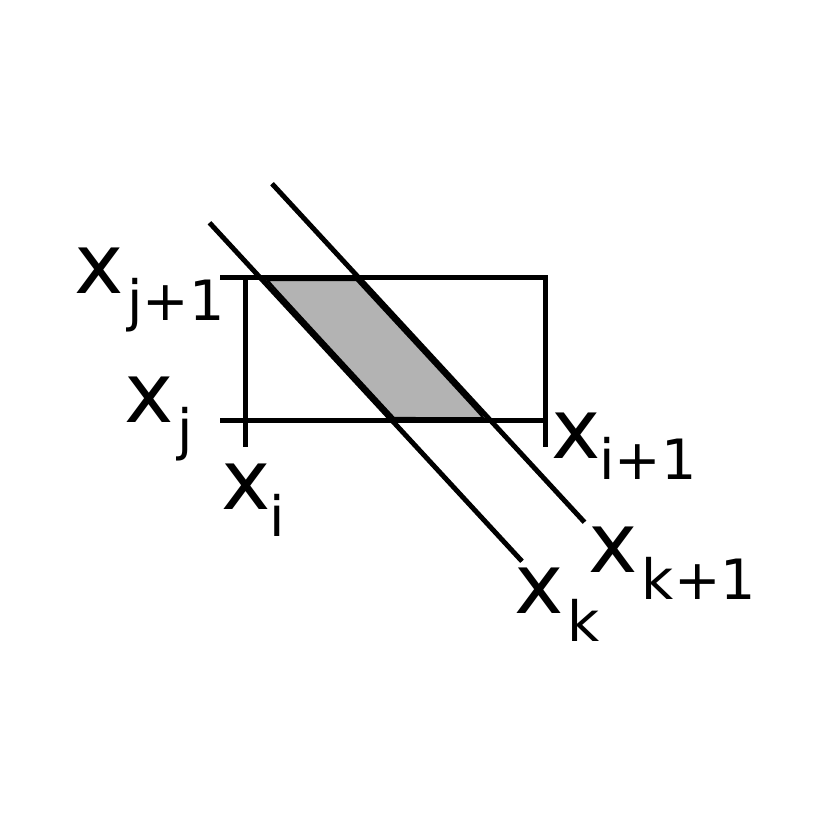}{x_{i} + x_{j+1} \leq 2x_{k} \leq 2x_{k+1} \leq x_{i+1} + x_{j}}{(x_{j+1} - x_{j})(2x_{k+1} - 2x_{k})}
\polyRegBox{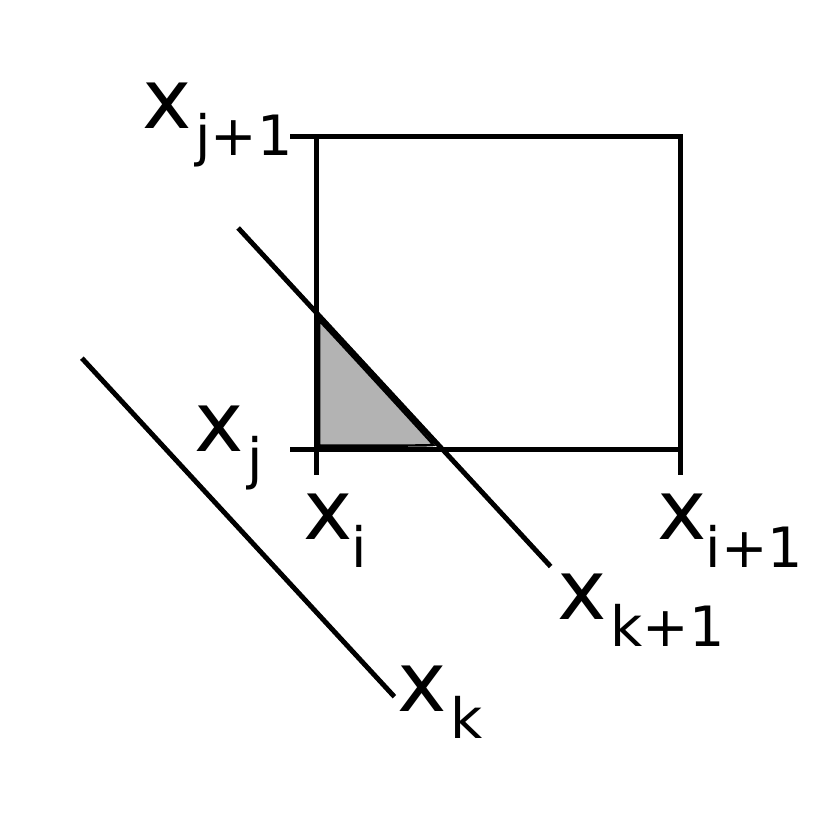}{2x_{k} \leq x_{i} + x_{j} \leq 2x_{k+1} \leq \min(x_{i} + x_{j+1}, x_{i+1} + x_{j})}{(2x_{k+1} - x_{i} - x_{j})^2/2}
\polyRegBox{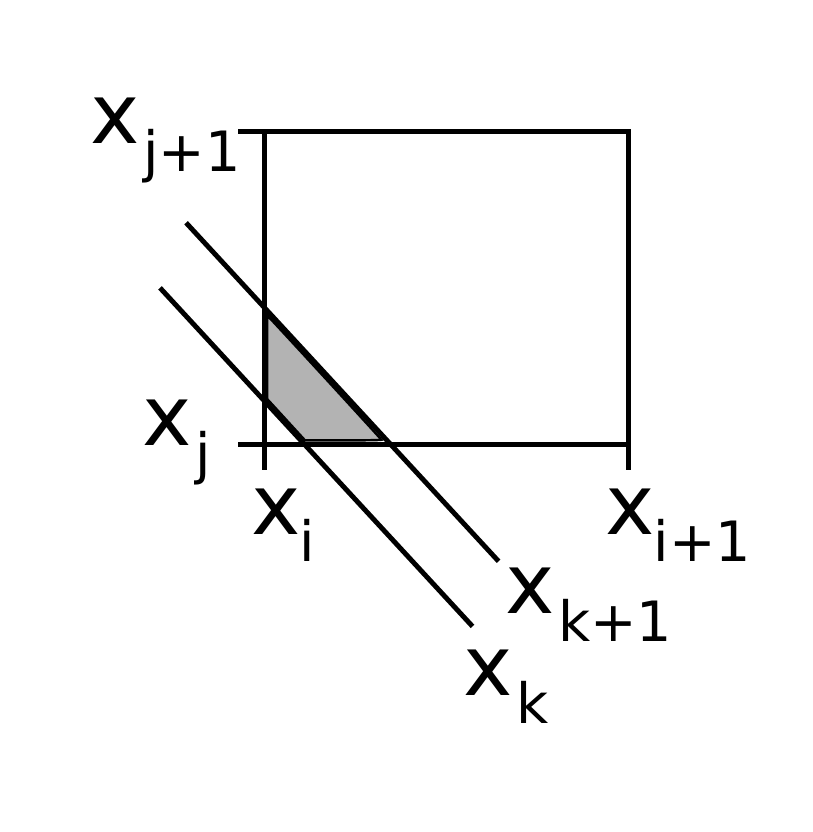}{x_{i} + x_{j} \leq 2x_{k} \leq 2x_{k+1} \leq \min(x_{i} + x_{j+1}, x_{i+1} + x_{j})}{(2x_{k+1} - x_{i} - x_{j})^2/2 - (2x_{k} - x_{i} - x_{j})^2/2}
\polyRegBox{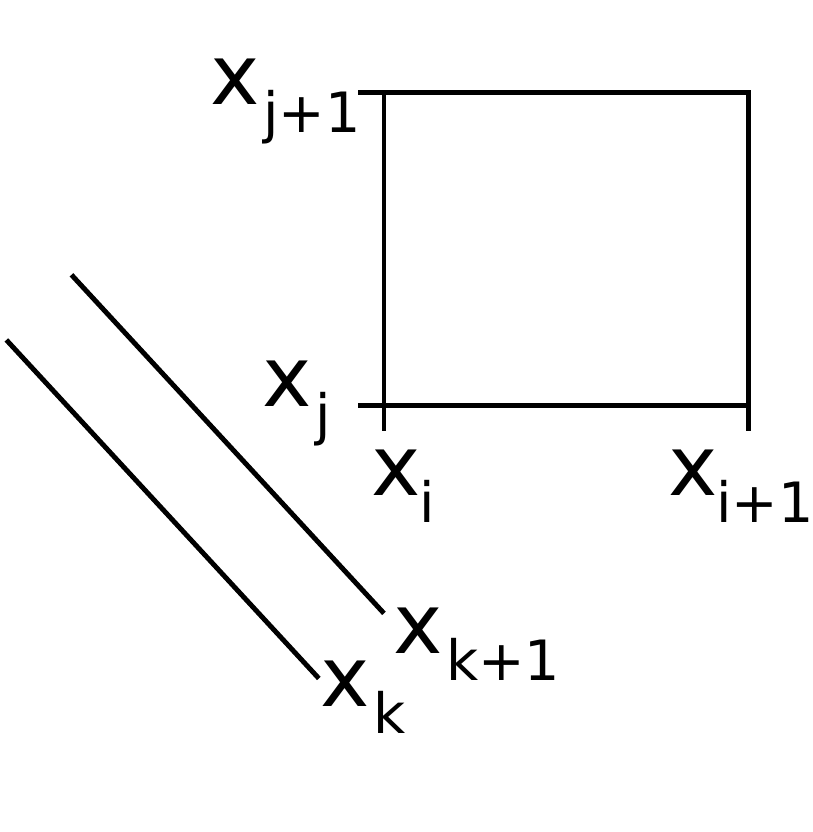}{2x_{k+1} \leq x_{i} + x_{j}}{0}
\polyRegBox{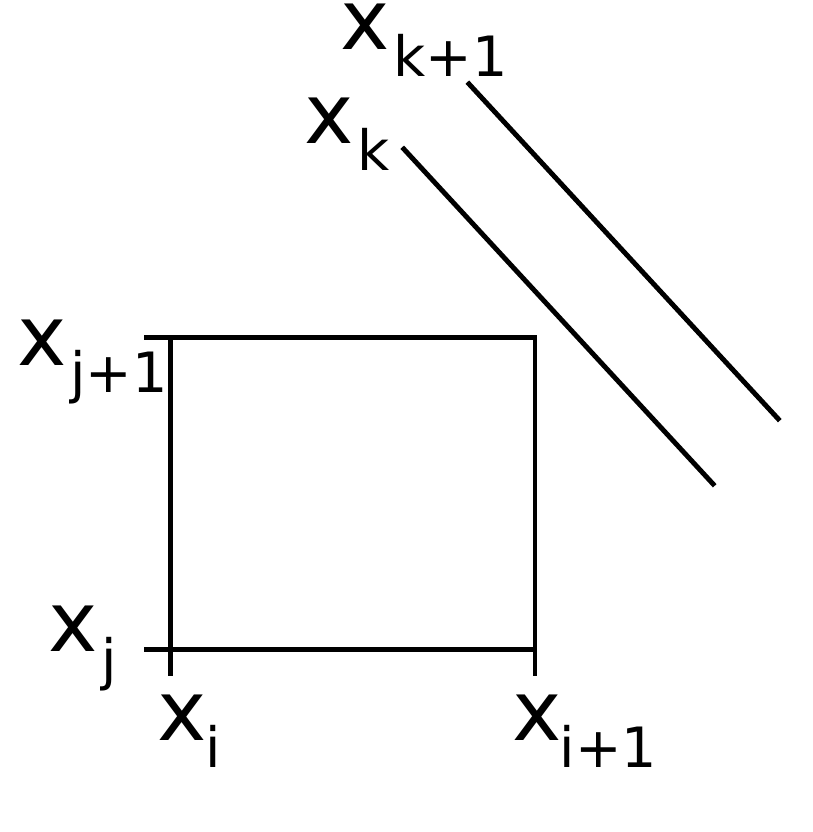}{x_{i+1} + x_{j+1} \leq 2x_{k}}{0}
}



\end{document}